\documentclass[12pt,a4paper,dvipsnames]{article}

\usepackage[utf8]{inputenc}

\usepackage[english]{babel}
\usepackage{amsmath}
\usepackage{amsthm}
\usepackage{amssymb}
\usepackage{geometry}
\usepackage{a4wide}
\usepackage{bm}
\usepackage{microtype}
\usepackage{mathtools}
\usepackage{csquotes}

\usepackage{enumitem}
\setlist[enumerate]{itemsep=0mm,parsep=2mm}

\usepackage{xcolor}
\newcommand\myshade{85}
\colorlet{myurlcolor}{Aquamarine}
\usepackage{hyperref}
\hypersetup{
  linkcolor  = black,
  citecolor  = black,
  urlcolor   = myurlcolor!\myshade!black,
  colorlinks = true,
}
\usepackage[capitalise, noabbrev, nameinlink]{cleveref}
\crefname{equation}{}{}
\usepackage{tikz, graphicx, standalone, caption, subcaption, wrapfig}
\usetikzlibrary{calc,decorations.text}

\usepackage{subcaption}

\def\normaledge{1.2}
\definecolor{edgeblack}{rgb}{0.25,0.25,0.25}
\definecolor{vertexblack}{rgb}{0.2,0.2,0.2}

\usepackage{chngcntr}
\usepackage{apptools}
\AtAppendix{\counterwithin{theorem}{subsection}}

\theoremstyle{definition}
\newtheorem{theorem}{Theorem}[section]

\newtheorem{lemma}[theorem]{Lemma}
\newtheorem*{lemma*}{Lemma}
\newtheorem*{conjecture*}{Conjecture}
\newtheorem*{lemma''*}{``Lemma''}
\newtheorem{claim}[theorem]{Claim}
\newtheorem*{claim*}{Claim}
\newtheorem{corollary}[theorem]{Corollary}

\newtheorem{conjecture}[theorem]{Conjecture}

\DeclareMathOperator{\rk}{rk}

\DeclareMathOperator{\supp}{supp}

\DeclareMathOperator{\grn}{grn*}

\newcommand{\RR}{\mathbb{R}}

\newcommand{\QQ}{\mathbb{Q}}

\newcommand{\norm}[1]{\left\lVert#1\right\rVert}

\newcommand{\twosum}{\oplus_2}

\newcommand{\R}{\mathbb{R}}
\newcommand{\Rd}{\mathcal{R}_d}

\Crefname{conjecture}{Conjecture}{Conjectures}

\begin{document}

\title{\bf Minimally globally rigid graphs}

\author{D\'aniel Garamv\"olgyi\thanks{Department of Operations Research, ELTE E\"otv\"os Lor\'and University, and the ELKH-ELTE Egerv\'ary Research Group
on Combinatorial Optimization, E\"otv\"os Lor\'and Research Network (ELKH),
P\'azm\'any P\'eter s\'et\'any 1/C, 1117 Budapest, Hungary.
e-mail: {\tt daniel.garamvolgyi@ttk.elte.hu }}\and
Tibor Jord\'an\thanks{Department of Operations Research, ELTE E\"otv\"os Lor\'and University, and the ELKH-ELTE Egerv\'ary Research Group
on Combinatorial Optimization, E\"otv\"os Lor\'and Research Network (ELKH),
P\'azm\'any P\'eter s\'et\'any 1/C, 1117 Budapest, Hungary.
e-mail: {\tt tibor.jordan@ttk.elte.hu}}
}

\date{September 6, 2022}

\maketitle

\begin{abstract}
A graph $G = (V,E)$ is globally rigid in $\R^d$ if for any generic placement $p : V \rightarrow \R^d$ of the vertices, the edge lengths $\norm{p(u) - p(v)}, uv \in E$ uniquely determine $p$, up to congruence.
In this paper we consider minimally globally rigid graphs, in which the deletion of an arbitrary edge destroys
global rigidity. We prove that if $G=(V,E)$ is minimally globally rigid in $\R^d$ on at least $d+2$ vertices, then $|E|\leq (d+1)|V|-\binom{d+2}{2}$. This implies that the minimum degree of $G$ is at most $2d+1$. We also show that the only graph in which the upper bound on the number of edges is attained is the complete graph $K_{d+2}$. It follows that every minimally globally rigid graph in $\R^d$ on at least $d+3$ vertices is flexible in $\R^{d+1}$.

As a counterpart to our main result on the sparsity of minimally globally rigid graphs, we show that in two dimensions, dense graphs always contain nontrivial globally rigid subgraphs.
More precisely, if some graph $G=(V,E)$ satisfies $|E|\geq 5|V|$, then $G$ contains a subgraph on at least seven vertices that is globally rigid in $\R^2$. If the well-known ``sufficient connectivity conjecture'' is true, then our methods also extend to higher dimensions. 

Finally, we discuss a conjectured strengthening of our main result, which states that if a pair of vertices $\{u,v\}$ is linked in $G$ in $\RR^{d+1}$, then $\{u,v\}$ is globally linked in $G$ in $\RR^d$. We prove this conjecture in the $d=1,2$ cases, along with a variety of related results.

\end{abstract}

\section{Introduction}

In this paper we investigate minimally globally rigid graphs. Our main result is an affirmative answer to the following conjecture made by the second author in \cite{handbook} (see also \cite[Conjecture 9]{jordan_extremal}).

\begin{conjecture}\cite[Conjecture 63.2.22]{handbook}\label{conjecture:minimallygloballyrigid}
Let $d$ be a positive integer and let $G = (V,E)$ be a graph on at least $d+2$ vertices that is minimally globally rigid in $\RR^d$. Then
\begin{enumerate}[label=\textit{(\alph*)}]
    \item $|E| \leq (d+1)|V| - \binom{d+2}{2}$ and
    \item the minimum degree of $G$ is at most $2d+1$.
\end{enumerate}
\end{conjecture}
\noindent 
Note that part \textit{(b)} follows from part \textit{(a)} by a simple counting argument. We also show that equality in part \textit{(a)} can only hold when $G$ is the complete graph on $d+2$ vertices. The motivation for studying minimally globally rigid graphs is twofold. First, studying the minimal elements of a graph family, and in particular giving bounds on the minimum degree in such graphs is often a key step in finding a constructive characterization for the given graph family. Second, the bound on the number of edges in the above conjecture hints at a possible connection between global rigidity in $\RR^d$ and rigidity in $\RR^{d+1}$ which we shall discuss below.

We briefly introduce the basic notions of combinatorial rigidity theory. The rest of the definitions, as well as references, are given in the next section. Let $d$ be a positive integer. A \emph{(bar-and-joint) framework} in $\RR^d$ is a pair $(G,p)$, where $G = (V,E)$ is a graph and $p : V \rightarrow \RR^d$ is a function that maps the vertices of $G$ into Euclidean space. We also say that $(G,p)$ is a \emph{realization} of $G$ in $\RR^d$. Two realizations $(G,p)$ and $(G,q)$ are \emph{equivalent} if the edge lengths coincide in the two frameworks, that is, if $\norm{p(u) - p(v)} = \norm{q(u) - q(v)}$ for every edge $uv \in E$. The realizations are \emph{congruent} if $\norm{p(u) - p(v)} = \norm{q(u) - q(v)}$ holds for every pair of vertices $u,v \in V$. A framework $(G,p)$ in $\RR^d$ is \emph{globally rigid} if every equivalent framework $(G,q)$ in $\RR^d$ is congruent to $(G,p)$. As a local counterpart, we define $(G,p)$ to be \emph{rigid} if there is some $\varepsilon > 0$ such that every equivalent framework $(G,q)$ in $\RR^d$ such that $\norm{p(v) - q(v)} < \varepsilon$ for all $v \in V$ is congruent to $(G,p)$. Informally, a framework is rigid if it is ``globally rigid in a small neighbourhood.''

A framework $(G,p)$ is \emph{generic} if the set of coordinates of $p(v),v \in V$ is algebraically independent over $\QQ$. It is known that for a given dimension $d \geq 1$, the rigidity and global rigidity of generic realizations of $G$ in $\RR^d$ is determined by $G$ itself, in the sense that if there is some generic realization of $G$ in $\RR^d$ that is rigid (resp.\ globally rigid), then every generic realization in $\RR^d$ is rigid (resp.\ globally rigid). 
We say that $G$ is \emph{rigid} (resp.\ \emph{globally rigid}) in $\RR^d$ if generic realizations of $G$ in $\RR^d$ are rigid (resp.\ globally rigid). When the dimension is clear from the context, we shall simply say that $G$ is rigid (globally rigid, respectively). If $G$ is not rigid in $\RR^d$, then we say that it is \emph{flexible} in $\RR^d$. 

In the following, we give an outline of our results. In \cref{section:minimallygloballyrigid}, we consider \emph{minimally globally rigid graphs}: graphs that are globally rigid in $\RR^d$, but for which deleting any edge results in a graph that is not globally rigid in $\RR^d$. In particular, we prove part \textit{(a)} of the above conjecture (\cref{theorem:minimallygloballyrigid}).
The proof is based on a well-known ``linear algebraic'' characterization of globally rigid graphs and a purely algebraic result about the rank of a linear combination of matrices (\cref{lemma:sumrank}). We also strengthen \cref{theorem:minimallygloballyrigid} by showing that the only graph for which equality holds in the above upper bound is the complete graph on $d+2$ vertices (\cref{theorem:sharp}). As is often the case, determining the tight examples requires a much more involved analysis than the upper bound itself. Our proof, while still linear algebraic in nature, relies on recent deep results about globally rigid graphs. 
We note that an infinite family of almost tight examples is given by the complete bipartite graphs $K_{d+1,n-d-1}$ for $n \geq \binom{d+2}{2}+1$.

One of the consequences of \cref{theorem:minimallygloballyrigid,theorem:sharp} is that if a graph on at least $d+3$ vertices is minimally globally rigid in $\RR^d$, then it is flexible in $\RR^{d+1}$. On the other hand, it is known that if a graph is rigid in $\RR^{d+1}$, then it is globally rigid in $\RR^d$. These results hint at an intriguing interplay between $d$-dimensional global rigidity and $(d+1)$-dimensional rigidity. In \cref{section:globallylinked} we propose a conjecture that would partly explain this phenomenon (\cref{conjecture:globallylinked}). Informally, it says that if adding an edge between a pair of vertices $\{u,v\}$ of a graph $G$ does not change the generic rigidity properties of $G$ in $\RR^{d+1}$ (that is, $\{u,v\}$ is \emph{linked} in $G$ in $\RR^{d+1}$), then adding the same edge does not change the generic global rigidity properties of $G$ in $\RR^{d}$ (that is, $\{u,v\}$ is \emph{globally linked} in $G$ in $\RR^d$). We prove this conjecture in the $d \leq 2$ case, as well as a variety of related results. 

\cref{theorem:minimallygloballyrigid} shows that minimally globally rigids have, in a sense, few edges. In \cref{section:globallyrigidsubgraphs} we discuss a similar question: do graphs with many edges necessarily have nontrivial globally rigid subgraphs? Here by nontrivial we mean that the subgraph has at least $d+2$ vertices. The existence of such subgraphs was also recently considered in \cite{MLT}, where it was linked to the so-called maximum likelihood threshold of the graph. We concentrate on the $d=2$ case and show that a graph $G = (V,E)$ on at least seven vertices with $|E| \geq 5|V| - 14$ always has a subgraph on at least seven vertices that is (redundantly) globally rigid in $\RR^2$.

We conclude with \cref{section:concluding}, where we put \cref{theorem:minimallygloballyrigid} into context by recalling related results about minimally $(d+1)$-connected graphs and minimally vertex-redundantly rigid graphs in $\RR^d$. It turns out that in each of these cases the tight upper bound on the number of edges in a minimal graph is approximately $(d+1)|V|$. We believe that a similar upper bound holds for the number of edges in a minimally redundantly rigid graph in $\RR^d$.

\section{Preliminaries}\label{section:preliminaries}

Throughout this section, let $d$ be a fixed positive integer.

\subsection{Stresses and the rigidity matroid}

The rigidity matroid of a graph $G$ is a matroid defined on the edge set
of $G$ which reflects the rigidity properties of all generic realizations of
$G$.
For a general introduction to matroid theory we refer the reader to \cite{oxley}. For a detailed treatment of the $2$-dimensional rigidity matroid, see \cite{Jmemoirs}.

Let $(G,p)$ be a realization of a graph $G=(V,E)$ in $\RR^d$.
The \emph{ rigidity matrix} of the framework $(G,p)$
is the matrix $R(G,p)$ of size
$|E|\times d|V|$, where, for each edge $uv\in E$, in the row
corresponding to $uv$,
the entries in the $d$ columns corresponding to vertices $u$ and $v$ contain
the $d$ coordinates of
$(p(u)-p(v))$ and $(p(v)-p(u))$, respectively,
and the remaining entries
are zeros. 
The rigidity matrix of $(G,p)$ defines
the \emph{rigidity matroid}  of $(G,p)$ on the ground set $E$
by linear independence of rows of the
rigidity matrix. It is known that any pair of generic frameworks
$(G,p)$ and $(G,q)$ have the same rigidity matroid.
We call this the $d$-dimensional \emph{rigidity matroid}
${\cal R}_d(G)=(E,r_d)$ of $G$.

We denote the rank of ${\cal R}_d(G)$ by $r_d(G)$.
A graph $G=(V,E)$ is \emph{$\Rd$-independent} if $r_d(G)=|E|$ and it is an \emph{$\Rd$-circuit} if it is not $\Rd$-independent but every proper 
subgraph $G'$ of $G$ is $\Rd$-independent. We note that in the literature such graphs are sometimes called $M$-independent in $\RR^d$ and $M$-circuits in $\RR^d$, respectively. 
An edge $e$ of $G$ is an \emph{$\Rd$-bridge in $G$}
if  $r_d(G-e)=r_d(G)-1$ holds. Equivalently, $e$ is an $\Rd$-bridge in $G$ if it is not contained in any subgraph of $G$ that is an $\Rd$-circuit.

The following characterization of rigid graphs is due to Gluck.

\begin{theorem}\label{theorem:gluck}
\cite{Gluck}
\label{combrigid}
Let $G=(V,E)$ be a graph with $|V|\geq d+1$. Then $G$ is rigid in $\RR^d$
if and only if $r_d(G)=d|V|-\binom{d+1}{2}$.
\end{theorem}
A graph is \emph{minimally rigid} in $\RR^d$ if it is rigid in $\RR^d$ but deleting any edge results in a flexible graph. By \cref{theorem:gluck}, minimally rigid graphs in $\RR^d$ on at least $d+1$ vertices have exactly $d|V| - \binom{d+1}{2}$ edges.

Stresses provide a dual viewpoint to rigidity that is often useful. Let $(G,p)$ be a framework. A vector $\omega \in \RR^E$ indexed by the edges of $G$ is an \emph{equilibrium stress} (or \emph{stress}, for short) of $(G,p)$ if $R(G,p)^T \omega = 0$. In other words, stresses of $(G,p)$ are the members of the cokernel of the rigidity matrix of $(G,p)$. It follows that the stresses of $(G,p)$ form a linear subspace of $\RR^E$ which we shall refer to as the \emph{space of stresses} of $(G,p)$. A graph is $\Rd$-independent if and only if every generic realization in $\RR^d$ is stress-free (i.e.\ the only stress for it is the zero vector). Similarly, a graph is an $\Rd$-circuit if and only if each generic realization in $\RR^d$ has a unique (up to scalar multiplication) nonzero stress, and this stress is everywhere nonzero. 

Let ${\cal M}$ be a matroid on ground set $E$. 
We can define a relation on the pairs of elements of $E$ by
saying that $e,f\in E$ are
equivalent if $e=f$ or there is a circuit $C$ of ${\cal M}$
with $\{e,f\}\subseteq C$.
This defines an equivalence relation. The equivalence classes are 
the \emph{connected components} of ${\cal M}$.
The matroid is \emph{connected} if there is only one equivalence class. 
A graph $G$ is \emph{$\Rd$-connected} if ${\cal R}_d(G)$ is connected.

$\mathcal{R}_2$-connected graphs played an important role in characterizing globally rigid graphs in $\RR^2$ (see \cref{theorem:2dimgloballyrigid} below). The following recent result of Steven J. Gortler and the authors suggest that $\Rd$-connected graphs may be useful in the study of $d$-dimensional global rigidity as well.

\begin{theorem}\label{theorem:Mconnected}
\cite[Theorem 3.5]{fullyreconstructible}
Let $G = (V,E)$ be a globally rigid graph in $\RR^d$ on $n\geq d+2$ vertices. 
Then $G$ is $\Rd$-connected.
\end{theorem}

We shall also use the following ``dimension dropping'' result from the same paper.

\begin{theorem}\label{theorem:Mconnecteddimensiondropping}\cite[Theorem 5.1]{fullyreconstructible}
Let $G$ be $\Rd$-connected. Then $G$ is $\mathcal{R}_{d'}$-connected for every $1 \leq d' < d$.
\end{theorem}

At some points we shall use the properties of the well-known \emph{coning operation}. The \emph{cone} of a graph $G$, denoted by $G^v$, is obtained from $G$ by adding a new
vertex $v$ and new edges from $v$ to each vertex of $G$. Coning provides a transfer between rigidity properties in $\RR^d$ and $\RR^{d+1}$.

\begin{theorem} \cite{whiteley_cones}\label{theorem:rigidconing}
A graph $G$ is 
rigid in $\RR^d$ ($\Rd$-independent, respectively) if and only if the cone of $G$
is rigid in $\R^{d+1}$ ($\mathcal{R}_{d+1}$-independent, respectively).
\end{theorem}

\begin{theorem} \cite{connelly.whiteley_2010} \label{theorem:globallyrigidconing}
A graph $G$ is globally rigid in $\RR^d$ if and only if the cone of $G$ is globally rigid in $\RR^{d+1}$.
\end{theorem}

\cref{theorem:rigidconing} also implies that an edge $e$ of $G$ is an $\Rd$-bridge if and only if $e$ is an $\mathcal{R}_{d+1}$-bridge in $G^v$, see \cite[Lemma 3]{KK}.

\subsection{Globally rigid graphs and stress matrices}

It follows from the definitions  that globally rigid graphs are rigid. 
The following much stronger necessary conditions of global rigidity 
 are due to Hendrickson \cite{hendrickson_1992}.
We say that a graph is \textit{redundantly rigid} in $\mathbb{R}^d$ if it
remains rigid in $\RR^d$ after deleting any edge. A graph is \emph{$k$-connected} for some positive integer $k$ if it has at least $k+1$ vertices and it remains connected after
deleting any set of less than $k$ vertices. 

\begin{theorem}\cite{hendrickson_1992}\label{theorem:hendrickson}
Let $G$ be a graph on $n \geq d+2$ vertices. Suppose that $G$ is globally rigid in $\RR^d$. Then $G$ is $(d+1)$-connected and redundantly rigid in $\RR^d$.
\end{theorem}

For $d \in \{1,2\}$, the conditions of \cref{theorem:hendrickson} are, in fact, sufficient for global rigidity. 
It is well-known that a graph is globally rigid in $\R^1$ if and only if it is $2$-connected (see, e.g., \cite[Theorem 63.2.6]{handbook}).
The characterization of $2$-dimensional global rigidity is as follows.

\begin{theorem} \cite{jackson.jordan_2005}
\label{theorem:2dimgloballyrigid}
Let $G$ be a graph on at least four vertices.
The following assertions are equivalent.
\begin{enumerate}[label=\textit{(\alph*)}]
\item $G$ is globally rigid in $\R^2$,
\item $G$ is $3$-connected and redundantly rigid in $\R^2$,
\item $G$ is $3$-connected and $\mathcal{R}_2$-connected.
\end{enumerate}
\end{theorem}

In contrast, in the $d\geq 3$ case the conditions of \cref{theorem:hendrickson}, together, are no longer
sufficent to imply global rigidity and the combinatorial characterization of globally rigid graphs in these dimensions is a major open question. However, there is a characterization based on the existence of stresses with particular properties which we recall now. 

Let $(G,p)$ be a realization in $\RR^d$ of the graph $G = (V,E)$, and let $\omega \in \RR^E$ be a stress of $(G,p)$. We define the \emph{stress matrix associated to $\omega$} as the matrix $\Omega \in \RR^{V \times V}$ with rows and columns indexed by the vertices of $G$ given by
\begin{equation*}
    \Omega_{uv} =
    \begin{cases}
      -\omega_{uv} & \text{if}\ uv \in E \\
      \sum_{w : uw \in E} \omega_{uw} & \text{if}\ u=v \\ 
      0 & \text{otherwise.}
    \end{cases}
\end{equation*}
In general, we say that $\Omega' \in \mathbb{R}^{V \times V}$ is a \emph{stress matrix of} $(G,p)$ if $\Omega'$ arises as the stress matrix associated to some stress $\omega' \in \RR^E$ of $(G,p)$.

If the set of points $p(v),v\in V$ is affinely spanning in $\RR^d$, then every stress matrix of $(G,p)$ has rank at most $|V| - d - 1$, see, e.g., \cite[Proposition 1.2]{connelly_2005}. In particular, this applies when $(G,p)$ is generic and $|V| \geq d+2$. Globally rigid graphs can be characterized by the existence of a stress matrix achieving this rank. The ``sufficiency'' part of the following theorem was proved by Connelly \cite{connelly_2005}, while ``necessity'' is due to Gortler, Healy and Thurston \cite{gortler.etal_2010}.  

\begin{theorem} \cite{connelly_2005,gortler.etal_2010} \label{theorem:stressmatrix}
Let $G = (V,E)$ be a graph on at least $d+2$ vertices. Then a generic realization of $G$ in $\RR^d$ is globally rigid if and only if it has a stress matrix of rank $|V|-d-1$. Moreover, if some generic realization in $\RR^d$ has a stress matrix of this rank, then every generic realization in $\RR^d$ has one.
\end{theorem}

We shall also need the following result.

\begin{theorem}\label{theorem:universallyrigid}\cite[Theorem 2.10]{universallyrigid}
Let $G = (V,E)$ be a globally rigid graph on at least $d+2$ vertices. Then there exists a generic realization of $G$ in $\RR^d$ that has a positive semidefinite stress matrix of rank $|V|-d-1$.
\end{theorem}

We recall the following result, which was already mentioned in the Introduction.

\begin{theorem}\label{theorem:dimensiondropping}\cite{jordan_extremal}
Let $G$ be a graph. If $G$ is rigid in $\RR^{d+1}$, then $G$ is globally rigid in $\RR^{d}$.
\end{theorem}

Finally, for some of our examples we shall need the following result on the global rigidity of complete bipartite graphs.
This theorem quickly follows from some well-known results on global rigidity and bipartite graphs, see \cite{bipartite1}. For a different proof, see \cite{bipartite2}.

\begin{theorem} \label{theorem:bipartite} \cite[Theorem 63.2.2]{handbook}
A complete bipartite graph $K_{m,n}$ is globally rigid in $\RR^d$ if and only if $m,n \geq d+1$ and $m+n \geq \binom{d+2}{2}+1$.
\end{theorem}

\subsection{Linked and globally linked pairs}

Let $G = (V,E)$ be a graph. 
A pair $\{u,v\}$ of vertices is \emph{linked} in $G$ in $\RR^d$ if $r_d(G+uv)=r_d(G)$ holds. By basic matroid theory, this is equivalent to the existence of an $\Rd$-circuit in $G+uv$ containing the edge $uv$.
In particular, if $G$ is rigid in $\RR^d$ and $G_0 = (V,E_0)$ is a minimally rigid spanning subgraph, then for every edge $e \in E - E_0$, there is a unique $\Rd$-circuit in $G_0+e$, which contains $e$. This is called the \emph{fundamental $\Rd$-circuit of $e$ with respect to $G_0$}. If $(G,p)$ is a generic framework in $\RR^d$, then for every $e \in E - E_0$, there is a unique (up to scalar multiple) stress $\omega_e$ of $(G,p)$ supported on the fundamental $\Rd$-circuit of $e$ with respect to $G_0$. The stresses $\omega_e, e \in E - E_0$ generate the space of stresses of $(G,p)$: they are linearly independent and
\begin{equation*}
|E - E_0| = |E| - \rk(R(G,p)) = \dim \ker(R(G,p)^T) 
\end{equation*}
is exactly the dimension of the space of stresses of $(G,p)$.  


A pair of vertices $\{u,v\}$ in a framework $(G,p)$ in $\RR^d$ is \emph{globally linked} in $(G,p)$ if for every equivalent framework $(G,q)$ in $\RR^d$ we have $\norm{p(u) - p(v)} = \norm{q(u) - q(v)}$.
This notion of global linkedness in $\RR^d$ is not a ``generic'' property for $d\geq 2$: 
a pair may be globally linked in some generic realization of $G$ without being globally linked in all generic realizations, see \cite{JJSz}. 
We define the pair $\{u,v\}$ to be \emph{globally linked in $G$} in $\RR^d$ if it is globally linked in all generic realizations of $G$ in $\RR^d$. Thus, $G$ is globally rigid in $\RR^d$ if and only if all pairs of vertices of $G$ are globally linked in $G$ in $\RR^d$.

This notion is well-understood for $d=1$. For a graph $G$ and two vertices $u,v\in V$ we use $\kappa(u,v;G)$ to denote the maximum number of pairwise
internally vertex-disjoint $u,v$-paths in $G$. It can be shown that $\{u,v\}$ is globally linked in $G$ in $\R^1$ if and only if $uv$ is an edge of $G$ or if $\kappa(u,v;G) \geq 2$.
In the $d \geq 2$ case, the characterization of globally linked pairs in graphs is an open question.
However, the special case when $d=2$ and the graph is $\mathcal{R}_2$-connected is also known.

\begin{theorem} \cite[Theorem 5.7]{JJSz}
\label{theorem:2dimgloballylinked}
Let $G = (V,E)$ be a graph. If $G$ is $\mathcal{R}_2$-connected, then a pair $\{u,v\}$ of vertices in $G$ is globally linked in $G$ in $\RR^2$ if and only if $\kappa(u,v;G)\geq 3$ holds.
\end{theorem}

\section{Minimally globally rigid graphs}\label{section:minimallygloballyrigid}

In this section we prove our main theorem, an upper bound on the number of edges in a minimally globally rigid graph in $\RR^d$. Our proof is based on the characterization of globally rigid graphs given by \cref{theorem:stressmatrix}. We note that the $d = 1,2$ cases can also be obtained by combinatorial methods. The one-dimensional case follows from a result of Mader \cite{Mader} stating that if a graph $G = (V,E)$ is minimally $2$-connected, then $|E| \leq 2|V| - 3$, with equality holding only if $G$ is a triangle. In the $d = 2$ case, an inductive characterization of globally rigid graphs in $\RR^2$ from \cite{jackson.jordan_2005} can be used to show that if $G$ is minimally globally rigid in $\RR^2$ with $|V| \geq 4$, then $|E| \leq 3|V| - 6$, see \cite[Theorem 11]{jordan_extremal}. A more careful analysis of the proof also gives that equality only holds for $G = K_4$.

The following result is the key lemma in our verification of \cref{conjecture:minimallygloballyrigid}. It says that if we can obtain a matrix of rank at least $r$ as a linear combination of some matrices $A_1,\ldots,A_k,$ where $k > r$, then we can also obtain a matrix of rank at least $r$ as a linear combination of some $r$-element subset of $\{A_1,\ldots,A_k\}$. We shall use the well-known fact that a multivariate polynomial over $\RR$ (or any infinite field) is the zero polynomial precisely if it is identically zero as a polynomial function (see, e.g., \cite[Chapter IV, Corollary 1.6]{lang_2005}).

\begin{lemma}\label{lemma:sumrank}
Let $A_1\ldots,A_k \in \mathbb{R}^{n \times n}$ be matrices and $1 \leq r < k$ a number. Suppose that there are scalars $t_1, \ldots, t_k \in \mathbb{R}$ such that $\sum_{i=1}^k{t_iA_i}$ has rank at least $r$. Then there is a subset $I \subseteq \{1,\ldots,k\}$ of size at most $r$ and scalars $t'_i,i \in I$ such that $\sum_{i \in I}t'_iA_i$ has rank at least $r$.
\end{lemma}
\begin{proof}
From the assumption it follows that $\sum_{i=1}^k t_iA_i$ has a nonsingular $r \times r$ submatrix. By replacing each $A_i$ with the corresponding submatrix we may suppose that the sum is nonsingular, so that $r = n$.

Consider the function $f : \mathbb{R}^k \rightarrow \RR$ defined by
\begin{equation*}
    (x_1,\ldots,x_k) \mapsto \det\bigl(\sum_{i=1}^k x_iA_i\bigr).
\end{equation*}
This is a polynomial function in $x_1,\ldots,x_k$, so we may consider $f$ as an element of $\mathbb{R}[x_1,\ldots,x_k]$. 
Since $\sum_{i=1}^k{t_iA_i}$ is nonsingular, $f(t_1,\ldots,t_k) \neq 0$, so in particular $f$ is not identically zero. Moreover, it has degree at most $r$. Consider any monomial $f_0$ of $f$ and let $x_i, i \in I$ denote the variables appearing in it. From the degree of $f$ we have that $|I| \leq r$. Let $g$ be the polynomial obtained from $f$ by setting the variables $x_i, i \notin I$ to zero. Then $g$ is nonzero, since it still contains the monomial $f_0$. It follows that the value of $g$ as a polynomial function is nonzero at some scalars $t'_i,i \in I$. This means precisely that $\sum_{i \in I}t'_iA_i$ is nonsingular, as required.
\end{proof}

If $r=k$ and $f$ contains a monomial other than $x_1 \cdots x_k$,  then the above argument also guarantees that we can find a suitable set $I$ with $|I| < r$. This observation implies the following corollary for the $r = k = n$ case.

\begin{corollary}\label{corollary:tight}
Let $A_1,\ldots,A_k \in \RR^{k \times k}$ be matrices and suppose that there are scalars $t_i,\ldots,t_k \in \RR$ such that $\sum_{i=1}^k t_iA_i$ is nonsingular. Then either there is a proper subset $I \subsetneq \{1,\ldots,k\}$ and scalars $t'_i,i \in I$ such that $\sum_{i \in I} t'_iA_i$ is nonsingular, or the polynomial $\det(\sum_{i=1}^kx_iA_i) \in \RR[x_1,\ldots,x_k]$ is equal to the monomial  $\alpha x_1\cdots x_k$ for some $\alpha \in \RR$. \qed
\end{corollary}

We note that the following ``discrete'' version of \cref{lemma:sumrank} seems to be open. Given a collection $A_1\ldots,A_k \in \mathbb{R}^{n \times n}$ of matrices such that $\sum_{i = 1}^k A_i$ has rank $r$, is it true that there is a subset $I \subseteq \{1,\ldots,k\}$ of size at most $r$ such that $\sum_{i \in I}A_i$ has rank at least $r$? The Cauchy-Binet formula can be used to give a quick proof for the case when each $A_i$ has rank one, see \cite{rankone}.

The upper bound on the number of edges in a minimally globally rigid graph follows easily from \cref{lemma:sumrank} and the characterization of global rigidity based on stress matrices.

\begin{theorem}\label{theorem:minimallygloballyrigid}
Let $G = (V,E)$ be a minimally globally rigid graph in $\RR^d$ on at least $d+2$ vertices. Then $|E| \leq (d+1)|V| - \binom{d+2}{2}$. 
\end{theorem}
\begin{proof}
Let $G_0 = (V,E_0)$ be a spanning minimally rigid subgraph of $G$ and let $(G,p)$ be a generic realization of $G$. Then for each $e \in E - E_0$, there is a unique stress $\omega_e$ of $(G,p)$ supported on the fundamental $\Rd$-circuit of $e$ with respect to $G_0$ with $\omega_e = 1$, and these stresses generate the space of stresses of $(G,p)$. Let $\Omega_e$ be the stress matrix corresponding to $\omega_e$.

For a contradiction, suppose that $|E| > (d+1)|V| - \binom{d+2}{2}$. Then 
\begin{equation*}
|E - E_0| > |V| - \binom{d+2}{2} + \binom{d+1}{2} = |V| - d -1. 
\end{equation*}
By \cref{theorem:stressmatrix}, $(G,p)$ has a stress matrix $\Omega$ of rank $|V|-d-1$. We can write $\Omega = \sum_{e \in E - E_0} t_e \Omega_e$ for some scalars $t_e \in \RR, e \in E - E_0$. Now by \cref{lemma:sumrank} there is a set $E' \subseteq E - E_0$ of size $|V| - d - 1$ and scalars $t'_e, e \in E'$ such that $\Omega' = \sum_{e \in E'}{t'_e \Omega_e}$ has rank at least $|V|-d-1$. Since the stress corresponding to $\Omega'$ is supported on $E' + E_0$, we can view $\Omega'$ as a stress matrix of $(H,p)$, where $H = (V,E' + E_0)$. By \cref{theorem:stressmatrix}, this implies that $H$ is a proper globally rigid spanning subgraph of $G$, contradicting the assumption that $G$ was minimally globally rigid.
\end{proof}

The above proof actually yields a slightly stronger conclusion: if $G$ is globally rigid in $\RR^d$ and has more than $(d+1)|V| - \binom{d+2}{2}$ edges, then for any minimally rigid spanning subgraph $G_0$ of $G$ there is a proper globally rigid spanning subgraph of $G$ that contains $G_0$. 

Our next goal is to improve on the upper bound given by \cref{theorem:minimallygloballyrigid} by showing that the
only graph for which equality holds is $K_{d+2}$. In order to prove this, we need some elementary observations from linear algebra.

\begin{lemma}\label{lemma:psdsum}
If $A,B \in \mathbb{R}^{n \times n}$ are symmetric, positive semidefinite matrices, then $\ker(A+B) = \ker(A) \cap \ker(B)$. 
\end{lemma}
\begin{proof}
Observe that for any positive semidefinite matrix $M$, we have $\ker(M) = \{v : v^TMv = 0\}$. Indeed, if $Mv = 0$, then $v^TMv = 0$ is immediate, and if $v^TMv = 0$, then taking a decomposition $M = X^TX$ of $M$ we see that $0 = v^TX^TXv = (Xv)^TXv = \norm{Xv}^2$, so that $Xv = 0$, and consequently $Mv=X^TXv = 0$. The statement follows easily from this observation and the fact that $v^TAv$ and $v^TBv$ are nonnegative for every $v \in \RR^n$
\end{proof}

\begin{lemma}\label{lemma:subspaceintersection}
Let $X_1,\ldots,X_k \subseteq \RR^n$ be linear subspaces and let $X = \cap_{i = 1}^k X_i$. Suppose that for every $j = 1,\ldots,k$, $\cap_{i \neq j} X_i \neq X$. Then $n > \dim X_i \geq k - 1 + \dim(X)$ holds for all $1 \leq i \leq k$.
\end{lemma}
\begin{proof}
We prove for $i = 1$. Notice that for every $j > 1$, $X_j$ is not contained in $\cap_{i < j}X_i$; indeed, otherwise $X = \cap_{i = 1}^k X_i = \cap_{i \neq j} X_i \neq X$, a contradiction. It follows that \begin{equation*}
    \RR^n \supsetneq X_1 \supsetneq X_1 \cap X_2 \supsetneq \ldots \supsetneq X_1 \cap \ldots \cap X_{k-1} \supsetneq X
\end{equation*} 
is a chain of nontrivial linear subspaces, which implies $n > \dim X_1 \geq k - 1 + \dim(X)$, as desired.
\end{proof}

\begin{lemma}\label{lemma:rankone}
Let $(G,p)$ be a realization of $G = (V,E)$ in $\RR^d$ and let $\omega$ be a stress of $(G,p)$ with corresponding stress matrix $\Omega$. Suppose that $\rk(\Omega) = 1$. Then $\supp(\omega) = \{e \in E: \omega(e) \neq 0\}$ induces a complete subgraph of $G$.
\end{lemma}
\begin{proof}
Since $\rk(\Omega) = 1$, there are vectors $x,y \in \RR^V$ such that $\Omega = xy^T$. Now if $uv \notin \supp(\omega)$, then $\Omega(uv) = 0$. It follows that $x_u$ or $y_v$ is zero, implying that the column of $\Omega$ corresponding to $u$ or $v$ is zero. This shows that if two vertices are both induced by $\supp(\omega)$, then they must be connected by an edge in $\supp(\omega)$.  
\end{proof}

Now we are ready to characterize the case of equality in \cref{theorem:minimallygloballyrigid}. Our approach will be similar to the one used in the proof of \cref{theorem:minimallygloballyrigid}. However, whereas in the latter we could use \cref{lemma:sumrank} to immediately derive a contradiction, here we can only rely on the weaker conclusion given by \cref{corollary:tight}, and thus we need to use a more involved argument.  

\begin{theorem}\label{theorem:sharp}
Let $G = (V,E)$ be a minimally globally rigid graph in $\RR^d$ on at least $d+2$ vertices. If $|E| = (d+1)|V| - \binom{d+2}{2}$, then $G$ is a complete graph on $d+2$ vertices.
\end{theorem}
\begin{proof}
For convenience, we let $k$ denote $|V| - d - 1$. By \cref{theorem:universallyrigid}, there is a generic realization $(G,p)$ of $G$ in $\RR^d$ which has a positive semidefinite stress matrix $\Omega$ of rank $k$. Let $\Omega'$ denote a $k \times k$ nonsingular submatrix of $\Omega$.
Consider a spanning minimally rigid subgraph $G_0 = (V,E_0)$ of $G$ and let $\Omega_1,\ldots,\Omega_k$ be the stress matrices of $(G,p)$ corresponding to the fundamental $\Rd$-circuits in $G$ with respect to $G_0$. Finally, let $\Omega'_i$ denote the submatrix of $\Omega_i$ corresponding to $\Omega'$. 

Since $\Omega_1,\ldots,\Omega_k$ generate the space of stress matrices of $(G,p)$, we have $\Omega = \sum_{i=1}^k t_i \Omega_i$ for some scalars $t_i \in \RR, i = 1,\ldots,k$. From the assumption that $G$ is minimally globally rigid, we must have $t_i \neq 0$ for all $i$. By possibly negating $\Omega_i$, we may suppose that $t_i > 0$ holds for each $i$. Thus $\Omega$ arises as a linear combination of $\Omega_i, i = 1,\ldots,k$ with positive coefficients.

Consider $f : \mathbb{R}^k \rightarrow \RR$ defined by
\begin{equation*}
    (x_1,\ldots,x_k) \mapsto \det(\sum_{i=1}^k x_i\Omega'_i).
\end{equation*}
From the assumption that $G$ is minimally globally rigid and \cref{corollary:tight} we have that $f(x_1,\ldots,x_k) = \alpha x_1x_2 \cdots x_k$ for some $\alpha \in \RR$. In particular, every linear combination of $\Omega_i, i = 1,\ldots,k$ with positive coefficients has rank $k$.

We first show that for every $t'_1,\ldots,t'_k > 0$, $\Psi = \sum_{i=1}^k t'_i \Omega_i$ is positive semidefinite.
Assume to the contrary that $\Psi$ has a negative eigenvalue $\lambda$ and consider the matrices $M_t = t\Omega + (1-t)\Psi$ for $0 \leq t \leq 1$. Let $f_t$ denote the characteristic polynomial of $M_t$. Since $(G,p)$ is a generic framework on at least $d+2$ vertices, it is affinely spanning, and consequently the stress matrix $M_t$ has a kernel of dimension at least $d+1$, for all $t$. It follows that $f_t(x) = x^{d+1}g_t(x)$ for some polynomial $g_t$. Also, since $M_t$ is symmetric, every root of $g_t$ is real. Now every root of $g_0$ is positive, while $g_1$ has $\lambda < 0$ as a root. Since the roots of a polynomial depend continuously on the coefficients, and the coefficients of $g_t$ are continuous in $t$, there must be some $0 \leq t_0 \leq 1$ such that $g_{t_0}$ has $0$ as a root. This implies that the kernel of $M_{t_0}$ is at least $d+2$-dimensional, so the rank of $M_{t_0}$ is at most $k - 1$. But this contradicts our observation that any linear combination of $\Omega_i, i = 1,\ldots,k$ with positive coefficients has rank $k$. Thus $\Psi$ is positive semidefinite, as claimed.

It follows that for every vector $v \in \RR^V$ and every $1 \leq i \leq k$, we have
\begin{equation*}
    v^T \Omega_i v = \lim_{t \rightarrow 0^+} v^T(\Omega_i + t\sum_{j \neq i}\Omega_j)v \geq 0,
\end{equation*}
since the matrix on the right side is positive semidefinite. This shows that $\Omega_i$ is positive semidefinite, for $i = 1,\ldots,k$.

Let $X_i = \ker \Omega_i$ for $i = 1,\ldots,k$. From \cref{lemma:psdsum} we have that $\cap_{i = 1}^k X_i = \ker(\Omega)$, and the fact that $G$ is minimally globally rigid implies that $\cap_{i \neq j} X_i \neq \ker(\Omega)$ for every $1 \leq j \leq k$. From \cref{lemma:subspaceintersection} we get that 
\begin{equation*}
|V| > \dim X_i \geq k - 1 + \dim(\ker(\Omega)) = |V| - 1, 
\end{equation*}
so $\rk(\Omega_i) = 1$. By \cref{lemma:rankone}, this means that the stress corresponding to $\Omega_i$ is supported on a complete graph. Since this stress was the nowhere-zero stress of a fundamental $\Rd$-circuit, and $K_{d+2}$ is the only complete $\Rd$-circuit, we obtain that each fundamental circuit corresponding to $G_0$ is a copy of $K_{d+2}$.

Recall that $G_0$ was an arbitrary spanning minimally rigid subgraph of $G$. It follows that every $\Rd$-circuit of $G$ is complete. Indeed, if $C$ is the edge set of an $\Rd$-circuit of $G$ and $e$ an edge in $C$, then $C-e$ spans an $\Rd$-independent graph, so we can extend it to a minimally rigid spanning subgraph of $G$. In this subgraph, the fundamental circuit corresponding to $e$ is $C$. By the previous arguments, this implies that $C$ is complete.

By \cref{theorem:Mconnected}, $G$ is $\Rd$-connected, so every pair of edges is contained in an $\Rd$-circuit of $G$. This means that every pair of edges is contained in a complete subgraph of $G$. But a graph with this property and without isolated vertices is necessarily complete: if $u,v \in V$ is a pair of vertices and $e,f$ are edges incident to $u$ and $v$, respectively, then the existence of a complete subgraph containing $e$ and $f$ shows that $uv \in E$.

We conclude that $G$ is a minimally globally rigid complete graph. It is well-known (and also follows from \cref{theorem:minimallygloballyrigid}) that this implies $|V| \leq d+2$. Since $G$ has at least $d+2$ vertices, we must have $G = K_{d+2}$, as desired.
\end{proof}

\cref{theorem:sharp} shows that the upper bound given by \cref{theorem:minimallygloballyrigid} on the number of edges in a minimally globally rigid graph is never tight for $|V| \geq d+3$. On the other hand, it is not far from being tight for all $d$ and $|V|$, as shown by the following example.\footnote{This was already noted in \cite{handbook}.}
Consider the bipartite graph $K_{d+1,n-d-1}$ 
with 
$n \geq \binom{d+2}{2} + 1$. This graph is globally rigid in $\RR^d$ by \cref{theorem:bipartite}, and moreover it is minimally globally rigid, since every edge is incident to a vertex of degree $d+1$.
A straightforward computation shows that this graph has \[n(d+1) - (d+1)^2 = n(d+1) - \binom{d+2}{2} - \binom{d+1}{2}\] edges. Thus, the gap between this example and the upper bound given by \cref{theorem:minimallygloballyrigid} is $\binom{d+1}{2}$.

As we already noted in the Introduction, \cref{theorem:minimallygloballyrigid} has the following corollary concerning the minimum degree in minimally globally rigid graphs.

\begin{corollary}
Let $G$ be minimally globally rigid in $\RR^{d}$ with minimum degree $\delta(G)$. Then we have
\[
\pushQED{\qed} 
d+1\leq \delta(G) \leq 2d+1. \qedhere
\popQED\]
\end{corollary}

It is known that the tight upper bound in $d=1,2$ dimensions is, in fact, $d+1$.
For $d\geq 3$ we can give examples of minimally globally rigid graphs $G$ in $\RR^d$ with $\delta(G)=d+2$ as follows.
In the $d=3$ case, consider a graph $G=(V,E)$ obtained from a 5-connected triangulation
by adding an edge (see \cref{figure:icosahedron}). It is a globally rigid in $\R^3$
by \cite[Theorem 7.1]{JT} and has $\delta(G)=5$.
Since it has $3|V|-5$ edges, $G-e$ is not redundantly rigid for any edge $e$, so $G$ is minimally globally rigid by \cref{theorem:hendrickson}.  
Examples in $d\geq 4$ dimensions can be obtained from $G$ by repeatedly applying the coning operation and using \cref{theorem:globallyrigidconing}.
It is an interesting research problem to find a best possible upper bound on the minimum
degree for $d\geq 3$. We believe that $2d+1$ is not tight.

\begin{figure}[t]
        \centering
        \begin{tikzpicture}[x = 1cm, y = 1cm, scale = 1]
            \tikzset{every node/.style={circle,draw=black, fill=vertexblack,minimum size=6pt,inner sep=0pt}}

            \node (left) at (-2,0) {};
            \node (right) at (2,0) {};
            \node (top) at (0,3.3) {};

            \node (middlebottom) at (0,0.4) {};
            \node (middle) at (0,.9) {};
            \node (middletop) at (0,1.7) {};
 
            \node (insideleft) at (-.3,1.3) {};
            \node (insideright) at (.3,1.3) {};

            \node (lefttop) at (-.7,1.5) {};
            \node (righttop) at (.7,1.5) {};

            \node (leftbottom) at (-.5,.8) {};
            \node (rightbottom) at (.5,.8) {};
            
            \draw [line width=\normaledge,color=edgeblack] (left) -- (right) -- (top) -- (left) -- (lefttop) -- (middletop) -- (righttop) -- (rightbottom) -- (middlebottom) -- (leftbottom) -- (lefttop) -- (insideleft) -- (middletop) -- (insideright) -- (insideleft) -- (middle) -- (insideright) -- (righttop) -- (top) -- (middletop) -- (top) -- (lefttop);
            \draw [line width=\normaledge,color=edgeblack] (insideleft) -- (leftbottom) -- (middle) -- (middlebottom) -- (right) -- (rightbottom) -- (insideright);
            \draw [line width=\normaledge,color=edgeblack] (leftbottom) -- (left) -- (middlebottom);
            \draw [line width=\normaledge,color=edgeblack] (right) -- (righttop);
            \draw [line width=\normaledge,color=edgeblack] (middle) -- (rightbottom);
            
            \draw [line width=\normaledge,color=red,dashed] (top) -- (insideleft);
            
        \end{tikzpicture}
        \caption{The skeleton of the icosahedron with an additional ``bracing'' edge. This graph is minimally globally rigid in $\RR^3$ and has minimum degree $5$.}
        \label{figure:icosahedron}
\end{figure}
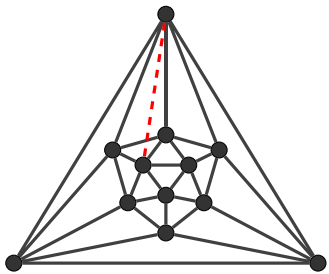

As another corollary to \cref{theorem:minimallygloballyrigid,theorem:sharp}, we have the following counterpart to \cref{theorem:dimensiondropping}.

\begin{corollary}\label{corollary:dimensiondropping}
Let $G$ be a graph. If $G$ is minimally globally rigid in $\RR^d$, then either $G$ is a complete graph on at most $d+2$ vertices, or $G$ is flexible in $\RR^{d+1}$. \qed
\end{corollary}

\section{Globally linked pairs and redundantly \texorpdfstring{$\Rd$}{Rd}-connected graphs}\label{section:globallylinked}

\cref{theorem:dimensiondropping} and \cref{corollary:dimensiondropping} suggest that there is an interplay between global rigidity in $\RR^d$ and rigidity in $\RR^{d+1}$ that we feel is not yet fully understood.
In this section we consider the following conjecture, which would somewhat clarify this connection; in particular, it would imply both \cref{theorem:dimensiondropping} and a strengthening of \cref{theorem:minimallygloballyrigid}.

\begin{conjecture}\label{conjecture:globallylinked}
Let $G = (V,E)$ be a graph and suppose that $\{u,v\}$ is linked in $G$ in $\RR^{d+1}$. Then $\{u,v\}$ is globally linked in $G$ in $\RR^d$.
\end{conjecture}

Our main result in this section is a positive answer to the $d=2$ case of \cref{conjecture:globallylinked} (see \cref{theorem:2dimconjecture} below). As we shall see, the proof relies on a ``dimension dropping'' result in the spirit of \cref{theorem:dimensiondropping} but concerning $\Rd$-connected graphs. Before giving the details, we consider how \cref{conjecture:globallylinked} relates to some previous results and conjectures.  

To see that \cref{conjecture:globallylinked} implies \cref{theorem:dimensiondropping} note that if a graph $G$ is rigid in $\RR^{d+1}$, then every pair of vertices is linked in $G$ in $\RR^{d+1}$. The conjecture would then imply that every pair of vertices is globally linked in $G$ in $\RR^d$, which is equivalent to $G$ being globally rigid in $\RR^d$. 

It also follows from basic matroid theory that a graph $G$ is $\Rd$-independent if and only if the the pair $\{u,v\}$ is not linked in $G-uv$ in $\RR^d$, for every edge $uv$ of $G$. Thus from \cref{conjecture:globallylinked} it would follow that if a graph is minimally globally rigid in $\RR^d$, then it is $\mathcal{R}_{d+1}$-independent. This would strengthen \cref{theorem:minimallygloballyrigid}: it implies that a minimally globally rigid graph is not only sparse, but ``everywhere sparse''. We note that the tight form of \cref{theorem:minimallygloballyrigid} given by \cref{theorem:sharp} (or its consequence \cref{corollary:dimensiondropping}) does not seem to follow from \cref{conjecture:globallylinked}.

By taking the contrapositive, \cref{conjecture:globallylinked} can be equivalently formulated in the following way: let $G = (V,E)$ be a graph and let $uv \in E$ be an edge. If $\{u,v\}$ is not globally linked in $G-e$, then $e$ is an $\mathcal{R}_{d+1}$-bridge in $G$. Restricting this to globally rigid graphs recovers the following conjecture from \cite{jordan_extremal}.

\begin{conjecture}\cite[Conjecture 13]{jordan_extremal}
\label{bridgeconj}
Let $G$ be globally rigid in $\R^d$
and suppose that $G-e$ is not.
Then $e$ is an $\mathcal{R}_{d+1}$-bridge in $G$.
\end{conjecture}

\cref{theorem:bridge} below from \cite{jordan_extremal} implies an affirmative
answer to
\cref{bridgeconj} 
in the special cases $d=1,2$. For completeness, we reproduce the proof. First, we recall a result from \cite{JJdress} that we shall use repeatedly throughout this section.

\begin{lemma}\cite[Lemma 2.5]{JJdress}\label{lemma:Mcircuits}
Let $G = (V ,E)$ be an $\Rd$-circuit. Then $G$ is $2$-connected. Furthermore, if $uv \in E$, then $\kappa(u,v;G) \geq d+1$.
\end{lemma}

\begin{theorem}\cite[Theorem 14]{jordan_extremal}
\label{theorem:bridge}
Let $G$ be a graph and
suppose that for some
edge $e=uv\in E$ one of the following holds:
\begin{enumerate}[label=\textit{(\alph*)}]
    \item $G$ is $(d+1)$-connected, but $G-e$ is not $(d+1)$-connected,
    \item $G$ is redundantly rigid in $\RR^d$, but $G-e$ is not redundantly rigid in $\R^d$.
\end{enumerate}
Then $e$ is an $\mathcal{R}_{d+1}$-bridge in $G$.
\end{theorem}

\begin{proof} 
\textit{(a)}
Since $G$ is $(d+1)$-connected but $G-e$ is not, we must have $\kappa(u,v;G) \leq d+1$. Now \cref{lemma:Mcircuits} implies that $e$ cannot belong to any $\mathcal{R}_{d+1}$-circuit of $G$, so it is indeed an $\mathcal{R}_{d+1}$-bridge in $G$.

\textit{(b)} Since $G-e$ is not redundantly rigid in $\R^d$, there is an edge $f\in
E$
for which $G-e-f$ is not rigid in $\R^d$. Since $G-f$ is rigid in $\RR^d$, $e$ must be an $\Rd$-bridge in 
$G-f$.
Let $w$ be an end-vertex of $f$ that is disjoint from $e$. Then $e$ must be an $\Rd$-bridge in $G-w$ as well, since $G-w$ is a subgraph of $G-f$.
Therefore, by \cref{theorem:rigidconing},
$e$ is an $\mathcal{R}_{d+1}$-bridge in $G^w$, where $G^w$ is the ``local cone'' $G + \{uw, u \in V\}$. Since $G$ is a subgraph of $G^w$, this implies
that $e$ is an $\mathcal{R}_{d+1}$-bridge in $G$, too.
\end{proof}

Since for $d \leq 2$, global rigidity in $\RR^d$ is equivalent to $(d+1)$-connectivity and redundant rigidity in $\RR^d$ (see \cref{theorem:2dimgloballyrigid}), \cref{theorem:bridge} implies a positive answer to the $d \leq 2$ case of \cref{bridgeconj}. 

\begin{theorem}
\label{bridgeconjspec}
Let $G$ be globally rigid in $\R^d$, where $d \in \{1,2\}$,
and suppose that $G-e$ is not.
Then $e$ is an $\mathcal{R}_{d+1}$-bridge in $G$.
\end{theorem}

Our next goal is to prove \cref{conjecture:globallylinked} for $d \in \{1,2\}$. The one-dimensional case follows from \cref{lemma:Mcircuits}: if $uv$ is an edge of $G$, then the statement is trivial; otherwise $uv$ is contained in an $\mathcal{R}_2$-circuit of $G+uv$, so $\kappa(u,v;G) \geq 2$ holds, which is equivalent to $\{u,v\}$ being globally linked in $\RR^1$.

The $2$-dimensional result follows quickly from the following lemma.
We say that a graph $G$ is \emph{redundantly $\Rd$-connected} if $G-e$ is $\Rd$-connected for every edge $e$ of $G$.

\begin{lemma}\label{theorem:2dimredundantlyMconnected}
Suppose that the graph $G$ is $\mathcal{R}_3$-connected. Then $G$ is redundantly $\mathcal{R}_2$-connected. 
\end{lemma}

\begin{theorem}\label{theorem:2dimconjecture}
Let $G = (V,E)$ be a graph and suppose that $\{u,v\}$ is linked in $G$ in $\RR^3$. Then $\{u,v\}$ is globally linked in $G$ in $\RR^2$.
\end{theorem}
\begin{proof}
We may assume that $uv\notin E$.
Since $\{u,v\}$ is linked in $G$ in $\RR^3$, there is an $\mathcal{R}_3$-circuit $C$ in $G+uv$ containing $uv$. By \cref{lemma:Mcircuits} 
we have $\kappa(u,v;C)\geq 4$, which implies $\kappa(u,v;C-uv)\geq 3$.
Furthermore,
from \cref{theorem:2dimredundantlyMconnected} we have that $C-uv$ is $\mathcal{R}_2$-connected. 
\cref{theorem:2dimgloballylinked} now implies that $\{u,v\}$ is globally linked in $C-uv$ in $\RR^2$, so, since $C-uv$ is a subgraph of $G$, $\{u,v\}$ is also globally linked in $G$ in $\RR^2$.
\end{proof}

Thus, it only remains to prove \cref{theorem:2dimredundantlyMconnected}. To this end, we shall need some facts about the $2$-sum operation, which we define next.
Let $G_1=(V_1,E_1)$ and $G_2 = (V_2,E_2)$ be two disjoint graphs with designated edges $u_iv_i\in E_i$, $i=1,2$.
The \emph{$2$-sum} of $G_1$ and $G_2$ along $u_1v_1$ and $u_2v_2$, denoted by $G_1 \twosum G_2$, is obtained from $G_1-u_1v_1$ and $G_2-u_2v_2$
by identifying $u_1$ with $u_2$, and identifying $v_1$ with $v_2$. The inverse operation of taking the $2$-sum, which is performed
along a {\em nonadjacent} separating vertex pair $\{u,v\}$, is called
\emph{$2$-separation}. If $\{u,v\}$ is an \emph{adjacent} separating vertex pair in a graph $G$, then the operation of removing $uv$ from $G$ and then performing a $2$-separation along $\{u,v\}$ is called \emph{cleaving} along $\{u,v\}$. See \cref{figure:2sum}.

\begin{figure}[ht]
        \centering
        \begin{tikzpicture}[x = 1cm, y = .9cm, scale = 1]
            \tikzset{every node/.style={circle,draw=black, fill=vertexblack,minimum size=6pt,inner sep=0pt}}

            \coordinate (aCenter) at (0,0) {};

            \node[label=above:$u$] (atop) at ($ (aCenter) + (0,1.5) $) {};
            \node[label=below:$v$] (abottom) at ($ (aCenter) + (0,-1.5) $) {};
            \node (alefttop) at ($ (aCenter) + (-1,.8) $) {};
            \node (aleftbottom) at ($ (aCenter) + (-1,-.8) $) {};
            \node (arighttop) at ($ (aCenter) + (1,.8) $) {};
            \node (arightbottom) at ($ (aCenter) + (1,-.8) $) {};
            
            \node[draw=none,fill=none] (aCaption) at ($ (aCenter) + (0,-2.5) $) {$G_1 \oplus_2 G_2$};
            
            \draw [line width=\normaledge,color=edgeblack] (aleftbottom) -- (atop) -- (alefttop) -- (aleftbottom) -- (abottom) -- (arightbottom) -- (arighttop) -- (atop) -- (arightbottom);
            \draw [line width=\normaledge,color=edgeblack] (alefttop) -- (abottom) -- (arighttop);
            
            \coordinate (bCenter) at (5,0) {};
            
            \node[label=above left:$u_1$] (btop1) at ($ (bCenter) + (-.25,1.5) $) {};
            \node[label=below left:$v_1$] (bbottom1) at ($ (bCenter) + (-.25,-1.5) $) {};
            \node[label=above right:$u_2$] (btop2) at ($ (bCenter) + (.25,1.5) $) {};
            \node[label=below right:$v_2$] (bbottom2) at ($ (bCenter) + (.25,-1.5) $) {};
            \node (blefttop) at ($ (bCenter) + (-1.25,.8) $) {};
            \node (bleftbottom) at ($ (bCenter) + (-1.25,-.8) $) {};
            \node (brighttop) at ($ (bCenter) + (1.25,.8) $) {};
            \node (brightbottom) at ($ (bCenter) + (1.25,-.8) $) {};
            
            \node[draw=none,fill=none] (bCaption1) at ($ (bCenter) + (-.9,-2.5) $) {$G_1$};
            \node[draw=none,fill=none] (bCaption2) at ($ (bCenter) + (.9,-2.5) $) {$G_2$};
            
            \draw [line width=\normaledge,color=edgeblack] (btop1) -- (blefttop) -- (bleftbottom) -- (bbottom1) -- (blefttop);
            \draw [line width=\normaledge,color=edgeblack] (bleftbottom) -- (btop1) -- (bbottom1);
            \draw [line width=\normaledge,color=edgeblack] (btop2) -- (brighttop) -- (brightbottom) -- (bbottom2) -- (brighttop);
            \draw [line width=\normaledge,color=edgeblack] (brightbottom) -- (btop2) -- (bbottom2);
            
            \coordinate (cCenter) at (10,0) {};
            
            \node[label=above:$u$] (ctop) at ($ (cCenter) + (0,1.5) $) {};
            \node[label=below:$v$] (cbottom) at ($ (cCenter) + (0,-1.5) $) {};
            \node (clefttop) at ($ (cCenter) + (-1,.8) $) {};
            \node (cleftbottom) at ($ (cCenter) + (-1,-.8) $) {};
            \node (crighttop) at ($ (cCenter) + (1,.8) $) {};
            \node (crightbottom) at ($ (cCenter) + (1,-.8) $) {};
            
            \node[draw=none,fill=none] (cCaption) at ($ (cCenter) + (0,-2.5) $) {$G_1 \oplus_2 G_2 + uv$};
            
            \draw [line width=\normaledge,color=edgeblack] (cleftbottom) -- (ctop) -- (clefttop) -- (cleftbottom) -- (cbottom) -- (crightbottom) -- (crighttop) -- (ctop) -- (crightbottom);
            \draw [line width=\normaledge,color=edgeblack] (clefttop) -- (cbottom) -- (crighttop);
            \draw [line width=\normaledge,color=edgeblack] (ctop) -- (cbottom);
            
            \tikzset{every node/.style={}}
            \def\myshift#1{\raisebox{1.5ex}}
            \draw [->,line width=\normaledge,transform canvas={yshift=.3cm},dashed,color=edgeblack,shorten <= 3pt, shorten >= 3pt,postaction={decorate,decoration={text along path,text align=center,text={|\myshift|$2$-separation}}}] (arighttop) to [bend left=25] (blefttop);
            
            \draw [->,line width=\normaledge,transform canvas={yshift=.3cm},dashed,color=edgeblack,shorten <= 3pt, shorten >= 3pt] (clefttop) to [bend right=25] (brighttop);
            
            \draw [draw=none,transform canvas={yshift=.3cm},postaction={decorate,decoration={text along path,text align=center,text={|\myshift|cleaving}}}] (brighttop) to [bend left=25] (clefttop);
            
            \def\myshift#1{\raisebox{-2.5ex}}            
            \draw [->,transform canvas={yshift=-.3cm},line width=\normaledge,dashed,color=edgeblack,shorten <= 3pt, shorten >= 3pt] (bleftbottom) to [bend left=25] (arightbottom);
            
            \draw [draw=none,transform canvas={yshift=-.3cm},postaction={decorate,decoration={text along path,text align=center,text={|\myshift|$2$-sum}}}] (arightbottom) to [bend right=25] (bleftbottom);            
            
        \end{tikzpicture}
        \vspace*{-9mm}
        \caption{The $2$-sum operation on two copies of $K_4$.}
        \label{figure:2sum}
\end{figure}
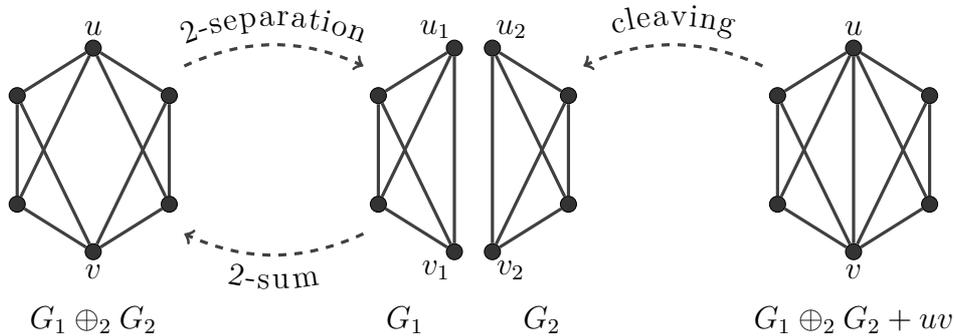

\begin{lemma}\label{lemma:circuit2sum}\cite[Lemma 10]{flexiblecircuits} (see also \cite{2sum})
Let $d$ be a positive integer and $G_i = (V_i,E_i), i = 1,2$ two disjoint graphs with $u_iv_i \in E_i$. Then $G_1 \twosum G_2$ is an $\Rd$-circuit if and only if $G_1$ and $G_2$ are $\Rd$-circuits.
\end{lemma}

Using \cref{lemma:circuit2sum}, we can prove similar results regarding $2$-sums of (redundantly) $\Rd$-connected graphs, \cref{lemma:Mconnected2sum,lemma:redundantlyMconnected2sum} below. Their proofs, which are easy, can be found in \cref{appendix:A}. We note that the $d = 2$ case of the following lemma can be found in \cite{jackson.jordan_2005}.

\begin{lemma}\label{lemma:Mconnected2sum}
Let $d$ be a positive integer and $G_i = (V_i,E_i), i = 1,2$ two disjoint graphs with designated edges $u_iv_i \in E_i$. The following are equivalent.
\begin{enumerate}[label=\textit{(\alph*)}]
    \item $G_1$ and $G_2$ are $\Rd$-connected,
    \item $G_1 \twosum G_2$ is $\Rd$-connected,
    \item $G_1 \twosum G_2 + uv$ is $\Rd$-connected, where $u$ and $v$ are the vertices obtained from $u_1,u_2$ and $v_1,v_2$, respectively, after identification.
\end{enumerate}
\end{lemma}

\begin{lemma}\label{lemma:redundantlyMconnected2sum}
Let $d$ be a positive integer and $G_i = (V_i,E_i), i = 1,2$ two disjoint graphs with designated edges $u_iv_i \in E_i$. Then the following hold.
\begin{enumerate}[label=\textit{(\alph*)}]
    \item If $G_1$ and $G_2$ are redundantly $\Rd$-connected, then $G_1 \twosum G_2$ is redundantly $\Rd$-connected.
    \item $G_1 \twosum G_2$ is redundantly $\Rd$-connected if and only if $G_1 \twosum G_2 + uv$ is redundantly $\Rd$-connected, where $u$ and $v$ are the vertices obtained from $u_1,u_2$ and $v_1,v_2$, respectively, after identification.
\end{enumerate}
\end{lemma}

We note that if $G_1 \twosum G_2$ is redundantly $\Rd$-connected, then $G_1-e$ is $\Rd$-connected for every edge $e$ other than $u_1v_1$, but it can happen that $G_1-u_1v_1$ is \emph{not} $\Rd$-connected. See \cref{fig:redundantlyMconnected} for a one-dimensional example.

\begin{figure}[b]
    \centering
    \begin{subfigure}[b]{0.4\linewidth}
        \centering
        \begin{tikzpicture}[x = 1cm, y = 1cm, scale = .8]

            \node[circle,draw=black, fill=vertexblack,minimum size=6pt,inner sep=0pt] (1) at (0,0) {};
            \node[circle,draw=black, fill=vertexblack,minimum size=6pt,inner sep=0pt,label=above:$u_1$] (2) at (1,1) {};
            \node[circle,draw=black, fill=vertexblack,minimum size=6pt,inner sep=0pt] (3) at (0,2) {};
            \node[circle,draw=black, fill=vertexblack,minimum size=6pt,inner sep=0pt] (4) at (-1,1) {};
            \node[circle,draw=black, fill=vertexblack,minimum size=6pt,inner sep=0pt,label=below:$v_1$] (5) at (1,-1) {};
            \node[circle,draw=black, fill=vertexblack,minimum size=6pt,inner sep=0pt] (6) at (0,-2) {};
            \node[circle,draw=black, fill=vertexblack,minimum size=6pt,inner sep=0pt] (7) at (-1,-1) {};

            \draw [line width=\normaledge,color=edgeblack] (1) -- (2);
            \draw [line width=\normaledge,color=edgeblack] (2) -- (3);
            \draw [line width=\normaledge,color=edgeblack] (3) -- (4);
            \draw [line width=\normaledge,color=edgeblack] (4) -- (1);
            \draw [line width=\normaledge,color=edgeblack] (1) -- (3);
            \draw [line width=\normaledge,color=edgeblack] (2) -- (4);
            
            \draw [line width=\normaledge,color=edgeblack] (1) -- (5);
            \draw [line width=\normaledge,color=edgeblack] (1) -- (6);
            \draw [line width=\normaledge,color=edgeblack] (1) -- (7);
            \draw [line width=\normaledge,color=edgeblack] (5) -- (6);
            \draw [line width=\normaledge,color=edgeblack] (5) -- (7);
            \draw [line width=\normaledge,color=edgeblack] (6) -- (7);

            \draw [line width=\normaledge,color=edgeblack] (5) -- (2);

            \node[circle,draw=black, fill=vertexblack,minimum size=6pt,inner sep=0pt] (a) at (3.5,0) {};
            \node[circle,draw=black, fill=vertexblack,minimum size=6pt,inner sep=0pt] (b) at (4.5,1) {};
            \node[circle,draw=black, fill=vertexblack,minimum size=6pt,inner sep=0pt] (c) at (3.5,2) {};
            \node[circle,draw=black, fill=vertexblack,minimum size=6pt,inner sep=0pt,label=above:$u_2$] (d) at (2.5,1) {};
            \node[circle,draw=black, fill=vertexblack,minimum size=6pt,inner sep=0pt] (e) at (4.5,-1) {};
            \node[circle,draw=black, fill=vertexblack,minimum size=6pt,inner sep=0pt] (f) at (3.5,-2) {};
            \node[circle,draw=black, fill=vertexblack,minimum size=6pt,inner sep=0pt,label=below:$v_2$] (g) at (2.5,-1) {};

            \draw [line width=\normaledge,color=edgeblack] (a) -- (b);
            \draw [line width=\normaledge,color=edgeblack] (a) -- (c);
            \draw [line width=\normaledge,color=edgeblack] (a) -- (d);
            \draw [line width=\normaledge,color=edgeblack] (b) -- (c);
            \draw [line width=\normaledge,color=edgeblack] (b) -- (d);
            \draw [line width=\normaledge,color=edgeblack] (c) -- (d);
            
            \draw [line width=\normaledge,color=edgeblack] (a) -- (e);
            \draw [line width=\normaledge,color=edgeblack] (a) -- (f);
            \draw [line width=\normaledge,color=edgeblack] (a) -- (g);
            \draw [line width=\normaledge,color=edgeblack] (e) -- (f);
            \draw [line width=\normaledge,color=edgeblack] (e) -- (g);
            \draw [line width=\normaledge,color=edgeblack] (f) -- (g);

            \draw [line width=\normaledge,color=edgeblack] (g) -- (d);
        \end{tikzpicture}
        \caption{}
    \end{subfigure}
    \hspace{2em}
    \begin{subfigure}[b]{0.4\linewidth}
        \centering
        \begin{tikzpicture}[x = 1cm, y = 1cm, scale = .8]

            \node[circle,draw=black, fill=vertexblack,minimum size=6pt,inner sep=0pt] (1) at (0,0) {};
            \node[circle,draw=black, fill=vertexblack,minimum size=6pt,inner sep=0pt,label=above:$u$] (2) at (1,1) {};
            \node[circle,draw=black, fill=vertexblack,minimum size=6pt,inner sep=0pt] (3) at (0,2) {};
            \node[circle,draw=black, fill=vertexblack,minimum size=6pt,inner sep=0pt] (4) at (-1,1) {};
            \node[circle,draw=black, fill=vertexblack,minimum size=6pt,inner sep=0pt,label=below:$v$] (5) at (1,-1) {};
            \node[circle,draw=black, fill=vertexblack,minimum size=6pt,inner sep=0pt] (6) at (0,-2) {};
            \node[circle,draw=black, fill=vertexblack,minimum size=6pt,inner sep=0pt] (7) at (-1,-1) {};

            \draw [line width=\normaledge,color=edgeblack] (1) -- (2);
            \draw [line width=\normaledge,color=edgeblack] (2) -- (3);
            \draw [line width=\normaledge,color=edgeblack] (3) -- (4);
            \draw [line width=\normaledge,color=edgeblack] (4) -- (1);
            \draw [line width=\normaledge,color=edgeblack] (1) -- (3);
            \draw [line width=\normaledge,color=edgeblack] (2) -- (4);
            
            \draw [line width=\normaledge,color=edgeblack] (1) -- (5);
            \draw [line width=\normaledge,color=edgeblack] (1) -- (6);
            \draw [line width=\normaledge,color=edgeblack] (1) -- (7);
            \draw [line width=\normaledge,color=edgeblack] (5) -- (6);
            \draw [line width=\normaledge,color=edgeblack] (5) -- (7);
            \draw [line width=\normaledge,color=edgeblack] (6) -- (7);

            \node[circle,draw=black, fill=vertexblack,minimum size=6pt,inner sep=0pt] (a) at (2,0) {};
            \node[circle,draw=black, fill=vertexblack,minimum size=6pt,inner sep=0pt] (b) at (3,1) {};
            \node[circle,draw=black, fill=vertexblack,minimum size=6pt,inner sep=0pt] (c) at (2,2) {};
            \node[circle,draw=black, fill=vertexblack,minimum size=6pt,inner sep=0pt] (d) at (1,1) {};
            \node[circle,draw=black, fill=vertexblack,minimum size=6pt,inner sep=0pt] (e) at (3,-1) {};
            \node[circle,draw=black, fill=vertexblack,minimum size=6pt,inner sep=0pt] (f) at (2,-2) {};
            \node[circle,draw=black, fill=vertexblack,minimum size=6pt,inner sep=0pt] (g) at (1,-1) {};

            \draw [line width=\normaledge,color=edgeblack] (a) -- (b);
            \draw [line width=\normaledge,color=edgeblack] (a) -- (c);
            \draw [line width=\normaledge,color=edgeblack] (a) -- (d);
            \draw [line width=\normaledge,color=edgeblack] (b) -- (c);
            \draw [line width=\normaledge,color=edgeblack] (b) -- (d);
            \draw [line width=\normaledge,color=edgeblack] (c) -- (d);
            
            \draw [line width=\normaledge,color=edgeblack] (a) -- (e);
            \draw [line width=\normaledge,color=edgeblack] (a) -- (f);
            \draw [line width=\normaledge,color=edgeblack] (a) -- (g);
            \draw [line width=\normaledge,color=edgeblack] (e) -- (f);
            \draw [line width=\normaledge,color=edgeblack] (e) -- (g);
            \draw [line width=\normaledge,color=edgeblack] (f) -- (g);
        \end{tikzpicture}
        \caption{}
    \end{subfigure}
    \caption{A pair of isomorphic graphs that are not redundantly $\mathcal{R}_1$-connected, but their $2$-sum is redundantly $\mathcal{R}_1$-connected.}
    \label{fig:redundantlyMconnected}
\end{figure}
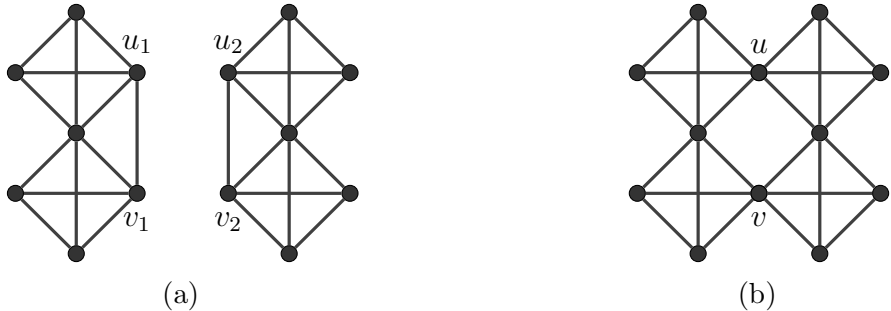

\begin{proof}[Proof of \cref{theorem:2dimredundantlyMconnected}]
We prove by induction on the number of vertices. The only $\mathcal{R}_3$-connected graph on at most five vertices is $K_5$, which is redundantly $\mathcal{R}_2$-connected. 

Now assume that $|V| \geq 6$. By \cref{theorem:Mconnecteddimensiondropping}, $G$ is $\mathcal{R}_2$-connected, and by \cref{lemma:Mcircuits} it is $2$-connected. If $G$ is also $3$-connected, then it is globally rigid in $\RR^2$ by \cref{theorem:2dimgloballyrigid}. In this case $G$ is, in fact, redundantly globally rigid in $\RR^2$, since otherwise \cref{bridgeconjspec} would imply that it has an $\mathcal{R}_3$-bridge, contradicting the assumption that $G$ is $\mathcal{R}_3$-connected. Using \cref{theorem:2dimgloballyrigid} again we get that $G$ is redundantly $\mathcal{R}_2$-connected, as desired. 

It remains to consider the case when $G$ is not $3$-connected. Let $\{u,v\}$ be a separating vertex pair in $G$ and let $G_1,G_2$ be the graphs obtained by a $2$-separation along $\{u,v\}$ if $u$ and $v$ are nonadjacent in $G$, and by cleaving along $\{u,v\}$ otherwise. Now \cref{lemma:Mconnected2sum} implies that $G_1$ and $G_2$ are $\mathcal{R}_3$-connected. By the induction hypothesis, $G_1$ and $G_2$ are redundantly $\mathcal{R}_2$-connected, and from \cref{lemma:redundantlyMconnected2sum} we get that $G$ is redundantly $\mathcal{R}_2$-connected. 
\end{proof}

As we saw, \cref{theorem:2dimconjecture} followed quickly from \cref{theorem:2dimredundantlyMconnected}. We close this section by briefly considering the $d$-dimensional analogue of the latter result. A positive answer to the following conjecture, which would be a strengthening of \cref{theorem:Mconnecteddimensiondropping}, may be useful in answering \cref{conjecture:globallylinked}.

\begin{conjecture}\label{conjecture:redundantlyMconnected}
Let $G$ be a graph and $d$ a positive integer. If $G$ is $\mathcal{R}_{d+1}$-connected, then $G$ is redundantly $\Rd$-connected.
\end{conjecture}

We first show that \cref{conjecture:redundantlyMconnected} is equivalent to the following conjecture, which is a natural analogue of \cref{theorem:bridge}.

\begin{conjecture}\label{conjecture:eqv}
Let $d$ be a positive integer and let $G$ be an $\mathcal{R}_d$-connected graph. Let $e$ be an edge of $G$. If $G-e$ is not $\mathcal{R}_d$-connected, then $e$ is an $\mathcal{R}_{d+1}$-bridge in $G$. 
\end{conjecture}

\begin{theorem}
\cref{conjecture:redundantlyMconnected,conjecture:eqv} are equivalent.
\end{theorem}
\begin{proof}
First, suppose that \cref{conjecture:eqv} holds and let $G$ be an $\mathcal{R}_{d+1}$-connected graph. By \cref{theorem:Mconnecteddimensiondropping}, $G$ is $\mathcal{R}_d$-connected. Since $G$ contains no $\mathcal{R}_{d+1}$-bridges, \cref{conjecture:eqv} implies that $G-e$ is $\mathcal{R}_d$-connected for every edge $e$ of $G$, as desired.

Now let us suppose that \cref{conjecture:redundantlyMconnected} holds and let $G$ be an $\mathcal{R}_d$-connected graph. We show that if some edge $e$ is not an $\mathcal{R}_{d+1}$-bridge in $G$, then $G-e$ is also $\mathcal{R}_d$-connected. 
The assumption on $e$ implies that there is an $\mathcal{R}_{d+1}$-circuit $C$ in $G$ that contains $e$. By \cref{conjecture:redundantlyMconnected}, $C-e$ is $\mathcal{R}_d$ connected. Thus, to show that $G-e$ is $\mathcal{R}_d$-connected, it suffices to show that for each edge $f$ of $G-e$, there is some edge $f'$ of $C-e$ such that $f$ and $f'$ are in the same connected component of $\mathcal{R}_d(G-e)$. But this is easy: since $G$ is $\mathcal{R}_d$-connected, there is an $\mathcal{R}_d$-circuit $C'$ in $G$ that contains $e$ and $f$, and since $C$ is $\mathcal{R}_d$-connected, there is some $\mathcal{R}_d$-circuit $C''$ in $C$ that contains $e$. By applying the strong circuit elimination axiom to $C',C'',f$ and $e$, we obtain an $\mathcal{R}_d$-circuit $C^*$ in $G-e$ that contains $f$ and also contains some edge of $C$. In particular, $f$ and $f^*$ are in the same connected component of $\mathcal{R}_d(G-e)$, as required. 
\end{proof}

The $d=1$ case of \cref{conjecture:redundantlyMconnected} follows from the fact that $\mathcal{R}_2$-connected graphs are redundantly rigid in $\RR^2$, and thus redundantly $2$-connected. A generalization of this argument shows that the conjecture is true in the special case when $G$ is rigid in $\RR^{d+1}$.

\begin{corollary}
If $G$ is rigid in $\RR^{d+1}$ and $\mathcal{R}_{d+1}$-connected, then $G$ is redundantly $\Rd$-connected.
\end{corollary}
\begin{proof}
Being $\mathcal{R}_{d+1}$-connected implies that there are no $\mathcal{R}_{d+1}$-bridges in $G$, so $G$ is redundantly rigid in $\RR^{d+1}$. Then by \cref{theorem:dimensiondropping} $G$ is redundantly globally rigid in $\RR^d$, and thus, by \cref{theorem:Mconnected}, redundantly $\Rd$-connected. 
\end{proof}

We can also prove the following weaker version of \cref{conjecture:redundantlyMconnected}: if a graph is $\mathcal{R}_{d+1}$-connected, then 
$G-e$ is $\Rd$-connected for {\em some} edge $e$ of $G$.
We shall use the following result. 

\begin{lemma}\label{lemma:redundantlybridgeless}
Suppose that the graph $G = (V,E)$ contains no $\mathcal{R}_{d+1}$-bridges. Then $G-v$ contains no $\Rd$-bridges, for every vertex $v \in V$.
\end{lemma}
\begin{proof}
If there were an $\Rd$-bridge $e \in E$ in $G-v$, then from \cref{theorem:rigidconing} it would follow that $e$ is an $\mathcal{R}_{d+1}$-bridge in $G^v$ where $G^v$ is the ``local cone'' $G + \{uv: u \in V\}$. But then $e$ would also be an $\mathcal{R}_{d+1}$-bridge in $G$, a contradiction.
\end{proof}

We also need the following result of Murty. 
We say that a matroid $\mathcal{M}$ on ground set $E$ is \emph{minimally connected} if it is connected but the restriction $M \setminus \{e\}$ is not connected, for all $e \in E$. A graph $G$ is \emph{minimally $\Rd$-connected} if $\mathcal{R}_d(G)$ is minimally connected.

\begin{lemma}\label{lemma:murty}\cite[Lemma 3.1]{Murty}
Let $\mathcal{M}$ be a minimally connected matroid of rank at least $2$. Then $\mathcal{M}$ contains a pair $\{e,f\}$ of non-bridges such that $e$ is a bridge in the restriction $\mathcal{M} \setminus \{f\}$ and $f$ is a bridge in $\mathcal{M} \setminus \{e\}$. 
\end{lemma}

Using terminology from matroid theory, \cref{lemma:murty} asserts that minimally connected matroids of rank at least $2$ contain a cocircuit of cardinality two.

\begin{theorem}
Let $G$ be a graph. If $G$ is $\mathcal{R}_{d+1}$-connected, then there is an edge $e$ of $G$ such that $G-e$ is $\Rd$-connected.
\end{theorem}
\begin{proof}
By \cref{theorem:Mconnecteddimensiondropping}, $G$ is $\Rd$-connected. Suppose for a contradiction that $G$ is minimally $\Rd$-connected. Then \cref{lemma:murty} applies and thus there is a pair $\{e,f\}$ of edges such that $f$ is an $\Rd$-bridge in $G-e$. Let $v$ be a vertex incident to $e$ but not to $f$. Then $f$ is an $\Rd$-bridge in $G-v$, contradicting \cref{lemma:redundantlybridgeless}.
\end{proof}

\section{Globally rigid subgraphs in dense graphs}\label{section:globallyrigidsubgraphs}

In this section we investigate whether sufficiently dense graphs always have ``nontrivial'' globally rigid subgraphs. In this context, we shall say that a graph that is globally rigid in $\RR^d$ is \emph{nontrivial} if it has at least $d+2$ vertices. More precisely, we consider the following problem: let $d$ be a fixed positive integer. Is it true that there is some integer $k_d$ such that for every graph $G = (V,E)$, if $|E| \geq k_d|V|$, then $G$ contains a nontrivial globally rigid subgraph in $\RR^d$?

The existence of nontrivial globally rigid subgraphs has recently been investigated in \cite{MLT}. In that paper, the authors defined the following graph parameter: given a graph $G$, let $\grn(G)$ denote the largest $d$ such that $G$ has a nontrivial globally rigid subgraph in $\RR^d$ (or $0$ if there is no such $d$). With this terminology, the above question amounts to giving a lower bound on $\grn$ in terms of the ``density'' $\frac{|E|}{|V|}$ of the graph. By the results in \cite{MLT}, this would also give a lower bound on the so-called maximum likelihood threshold of the graph.

We shall concentrate on the $d=2$ case. We first show how $k_2 \leq 9$ follows from previous results of Mader and Jackson and Jordán, respectively. Then we prove an analogue of Mader's theorem and use it to show that $k_2 \leq 5$.
We note that whether the same approach generalizes to the $d \geq 3$ case depends on the validity of the well-known ``sufficient connectivity conjecture''; see the discussion at the end of the section.

Our first ingredient is the following theorem of Mader regarding highly connected subgraphs in sufficiently dense graphs.

\begin{theorem} \cite{Mader}
\label{thm:mader}
Let $G=(V,E)$ be a graph and suppose that for some $k \geq 2$
we have $|V| \geq 2k-1$ and
\begin{equation*}
|E| > (2k-3)(|V|-k-1).
\end{equation*}
Then $G$ contains a $k$-connected subgraph.
\end{theorem}

To use \cref{thm:mader}, we need a result connecting sufficiently high vertex connectivity with global rigidity in $\RR^2$. This is provided by the following result. A graph is \emph{redundantly globally rigid} in $\RR^d$ if it remains globally rigid in $\RR^d$ after deleting any edge.

\begin{theorem} \cite{jackson.jordan_2005} \label{theorem:6connected}
If $G$ is $6$-connected, then $G$ is redundantly globally rigid in $\R^2$.
\end{theorem}

Combining the $k=6$ case of \cref{thm:mader} with \cref{theorem:6connected}, we get the following.

\begin{corollary}\label{corollary:globallyrigidsubgraphweak}
If $G = (V,E)$ is a graph on at least $11$ vertices such that 
\begin{equation*}
|E| > 9|V| - 63,
\end{equation*}
then $G$ contains a (redundantly) globally rigid subgraph in $\RR^2$ on at least $7$ vertices. \qed
\end{corollary}

The bound given in \cref{corollary:globallyrigidsubgraphweak} is far from being tight. 
In order to improve it, we prove a result similar to \cref{thm:mader} in the setting of mixed connectivity, which we introduce next.

Let $G=(V,E)$ be a graph. For a pair $S\subseteq V$ and $F\subseteq E$ we say that
$(S,F)$ is a \emph{mixed cut} of $G$ if $G-S-F$ is disconnected. A graph $G$ is called \emph{mixed $k$-connected}
if
\begin{equation*}
2|S|+|F| \geq k
\end{equation*}
holds for every mixed cut $(S,F)$ of $G$.
Note that $k$-connected graphs are mixed $k$-connected and mixed $k$-connected graphs are $\lceil \frac{k}{2} \rceil$-connected.

\cref{theorem:6connected} has the following strengthening in terms of mixed connectivity.

\begin{theorem} \cite{JJmixed}
\label{thm:JJb}
If $G$ is mixed $6$-connected, then $G$ is redundantly globally rigid in $\R^2$.
\end{theorem}

We can prove the following analogue of \cref{thm:mader} for graphs with sufficiently high mixed connectivity. The proof strategy is similar to  Mader's proof of \cref{thm:mader}.

\begin{theorem}
\label{maderext}
Let $G=(V,E)$ be a graph and suppose that for some even positive integer $k$
we have $|V| \geq k+1$ and
\begin{equation*}
|E| > (k-1)(|V|-\frac{k}{2}).
\end{equation*}
Then $G$ contains a mixed $k$-connected subgraph.
\end{theorem}

\begin{proof}
Let $H = (V',E')$ be a minimal subgraph of $G$ satisfying $|V'| \geq k+1$ and $|E'|> (k-1)(|V'|-\frac{k}{2})$.
We claim that $H$ is mixed $k$-connected.
First consider the case when $|V'|=k+1$. A graph on $k+1$ vertices has at most $\binom{k+1}{2}=\frac{k^2+k}{2}$ edges.
So if $H$ has at least $(k-1)(\frac{k}{2}+1)+1=\frac{k^2+k}{2}$ edges, then it is isomorphic to $K_{k+1}$,
which is mixed $k$-connected. 

Thus, we may assume that $|V'|\geq k+2$.
We must have $d_{H}(v)\geq k$ for all $v\in V(H)$, for otherwise we could replace $H$ by $H-v$, contradicting the minimality of $H$.
%
For a contradiction, suppose that $H$ has a mixed
cut $(S,F)$ with $2|S|+|F|<k$. 
Let $C_1$ be 
the union of some (but not all) connected components of $H-S-F$ 
and let 
$C_2=H-S-F-V(C_1)$. By symmetry, the following claim implies that $V(C_1)\cup S$ as well as $V(C_2)\cup S$ 
have cardinality at least $k+1$.

\begin{claim}
\label{cl:k}
Let $V_1=V(C_1)\cup S$. Then
$|V_1|\geq k+1$.
\end{claim}

\begin{proof}
For a contradiction, suppose that $|V_1|\leq k$.
Since $H$ is a simple graph with minimum degree at least $k$, each vertex in $C_1$ is adjacent to
at most $|V_1|-1$ vertices in $V_1$ and at least $k-|V_1|+1$ vertices in $C_2$. 
Since $C_1$ and $C_2$ are disconnected in $H - S - F$, this implies
\begin{equation*}
|F|\geq (|V_1|-|S|)(k-|V_1|+1),
\end{equation*}
which gives
\begin{align}
\label{count}
\begin{split}
2|S|+|F| &\geq (|V_1|-|S|)(k-|V_1|+1)+2|S| \\
    &= -|V_1|^2 + (k+|S|+1)|V_1| -k|S| + |S|.
\end{split}
\end{align}

We claim that the right-hand side of \cref{count} can be bounded from below by $k + |S|$. Indeed, consider the function 
\begin{equation*}
f: \RR \rightarrow \RR, \hspace{1em} x \mapsto -x^2 + (k + |S| + 1)x - k|S| + |S|;
\end{equation*}
note that the right-hand side of \cref{count} is just $f(|V_1|)$. It is straightforward to compute that $f(|S| + 1) = f(k) = k + |S|$. Since $f$ is concave and $|S| + 1 \leq |V_1| \leq k$, this implies that $f(|V_1|) \geq k + |S|$, as required. 

Combining this bound with \cref{count} gives $2|S| + |F| \geq k + |S| \geq k$. But this contradicts our assumption that $2|S| + |F| < k$.
\end{proof}

Now $H-F$ can be obtained as the union of two edge-disjoint subgraphs $H_1$ and $H_2$, on
vertex sets $V(C_i)\cup S$, $i=1,2$, respectively.
The minimality of $H$ and \cref{cl:k} imply that $m_i\leq (k-1)(n_i-\frac{k}{2})$ for $i=1,2$, where
$n_i$ and $m_i$ denote the number of vertices and edges of $H_i$, respectively.
Therefore we obtain, by using $|F|\leq k-1$ and $|S|\leq \frac{k}{2}-1$ (the latter of which follows from the assumption that $k$ is even), that
\begin{align*}
|E'|=m_1+m_2+|F| &\leq (k-1)(n_1+n_2-k) + |F| =(k-1)(|V'|+|S|-k) + |F| \\
&\leq (k-1)(|V'|+|S|-k+1) \leq (k-1)(|V'|-\frac{k}{2}),
\end{align*}
a contradiction. This shows that $H$ is indeed mixed $k$-connected, as claimed.
\end{proof}

For odd $k$,
applying \cref{maderext} to $k+1$ gives that if $|E| > k(|V|-\frac{k+1}{2})$, then $G$ has a
mixed $k$-connected subgraph. 

\cref{maderext,thm:JJb} imply the following strengthening of \cref{corollary:globallyrigidsubgraphweak}. 

\begin{corollary}\label{corollary:betterbound}
Let $G=(V,E)$ be a graph with $|V|\geq 7$ and $|E| \geq 5|V|-14$.
Then $G$ has a subgraph on at least seven vertices which is (redundantly) globally rigid in $\R^2$. \qed
\end{corollary}

We do not know whether the constant $5$ in \cref{corollary:betterbound} can be further improved. See \cref{figure:dense} for a family of graphs with approximately $\frac{5}{2}|V|$ edges and no nontrivial globally rigid subgraphs in $\RR^2$.


\begin{figure}
        \centering
        \begin{tikzpicture}[x = 1cm, y = 1cm, scale = 1]
            \tikzset{every node/.style={circle,draw=black, fill=vertexblack,minimum size=6pt,inner sep=0pt}}

            \node (top) at (-1.5,1.5) {};
            \node (bottom) at (1.5,1.5) {};
            
            \node (top1) at (-2.5,0) {};
            \node (bottom1) at (-1.5,0) {};

            \node (bottom2) at (2.5,0) {};
            \node (top2) at (1.5,0) {};

            \node (left) at (-.5,0) {};
            \node (right) at (.5,0) {};
            
            \draw [line width=\normaledge,color=edgeblack] (top) -- (top1) -- (bottom1) -- (bottom) -- (bottom2) -- (top2) -- (top) -- (bottom2);
            \draw [line width=\normaledge,color=edgeblack] (top1) -- (bottom) -- (top2);
            \draw [line width=\normaledge,color=edgeblack] (top) -- (bottom1);
            
            \draw [line width=\normaledge,color=edgeblack] (top) -- (left) -- (right) -- (bottom) -- (left);
            \draw [line width=\normaledge,color=edgeblack] (right) -- (top);
            
        \end{tikzpicture}
        \caption{Gluing $\ell$ copies of $K_4 - e$ ($\ell=3$ in the drawing) gives a graph with no nontrivial globally rigid subgraphs in $\RR^2$, on $2\ell + 2$ vertices and with $5\ell$ edges.}
        \label{figure:dense}
\end{figure}
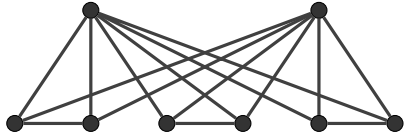

To close this section, we consider whether the above strategy for finding globally rigid subgraphs can be extended to the $d \geq 3$ case. We recall the following conjecture originally posed by Lovász and Yemini \cite{lovasz.yemini_1982}, which has now been open for over forty years.
\begin{conjecture}\cite[Conjecture 61.1.19]{handbook}\label{conjecture:sufficientconnectivity} If a graph is $d(d+1)$-connected, then it is rigid in $\RR^d$.
\end{conjecture}

A positive answer to \cref{conjecture:sufficientconnectivity} would imply that the graph parameter $\grn$ (defined at the start of this section) can be bounded from below by roughly $\sqrt{|E|/|V|}.$ We sketch the proof of this implication below.   

\begin{theorem}\label{theorem:grn}
If \cref{conjecture:sufficientconnectivity} is true, then for every graph $G = (V,E)$ we have \[\grn(G) \geq \left\lfloor\sqrt{\frac{|E|}{6 \cdot |V|}}\right\rfloor.\]
\end{theorem}
\begin{proof}
Let $G = (V,E)$ be a graph and let us define $\varrho = \frac{|E|}{|V|}$ for convenience. If $\varrho < 6$, then the right-hand side of the claimed inequality is zero and the inequality is trivially satisfied. Thus we may assume that $\varrho \geq 6$. 
 
Note that  \cref{thm:mader} implies that $G$ contains a $\lfloor\varrho/2\rfloor$-connected subgraph. Now take any $d \geq 1$ satisfying $d^2 \leq \lfloor{\varrho}/{6}\rfloor$. With this choice, we have
\[d(d+1) + 1 \leq 3d^2 \leq 3\left\lfloor\frac{\varrho}{6}\right\rfloor \leq \left\lfloor\frac{\varrho}{2}\right\rfloor,\]
and consequently $G$ contains a $(d(d+1)+1)$-connected subgraph. If \cref{conjecture:sufficientconnectivity} is true, then a theorem of Tanigawa \cite{tanigawa_2015} implies that every $(d(d+1) + 1)$-connected graph is globally rigid in $\RR^d$ (see \cref{section:concluding} for Tanigawa's result). Such a graph has at least $d(d+1) + 2 \geq d+2$ vertices. These observations together imply that
\[\grn(G) \geq \left\lfloor\sqrt{\left\lfloor\frac{\varrho}{6}\right\rfloor}\right\rfloor = \left\lfloor\sqrt{\frac{\varrho}{6}}\right\rfloor,\]as required.
\end{proof}

Consider the bipartite graph $K_{n,n}$. By \cref{theorem:bipartite} this graph is globally rigid in $\RR^d$ if and only if $2n \geq \binom{d+2}{2}$, which implies that $\grn(K_{n,n}) = O(\sqrt{n})$. On the other hand, the ``density'' of $K_{n,n}$ is $\frac{n^2}{2n} = O(n)$. This shows that the lower bound given by \cref{theorem:grn} would be asymptotically tight. 

\section{Concluding remarks}\label{section:concluding}

\subsection{Vertex redundant rigidity and \texorpdfstring{$(d+1)$}{(d+1)}-connectivity}


We say that a graph $G$ is \emph{ vertex-redundantly rigid} in $\R^d$ if
$G-v$ is rigid in $\R^d$ for all $v\in V(G)$. By a result of
Tanigawa \cite{tanigawa_2015} vertex-redundant rigidity in $\R^d$ implies
global rigidity in $\R^d$. 
On the other hand, global rigidity in $\RR^d$ implies
$(d+1)$-connectivity (cf. \cref{theorem:hendrickson}).


It is interesting to note that the tight upper bounds for the
number of edges of a graph $G=(V,E)$ that is minimal with respect to any of these three
properties are the same up to a small constant. A well-known result of Mader \cite{Mader2} states that for all $d\geq 0$ and $|V|\geq 3d+1$,
the number of edges in a minimally $(d+1)$-connected graph on $|V|$ vertices
satisfies \[|E|\leq (d+1)|V|-(d+1)^2.\] 
This bound is tight, as shown by the
graph $K_{n-d-1,d+1}$.
In the case of global rigidity we have \cref{theorem:minimallygloballyrigid},
which gives the upper bound
\begin{equation*}
|E| \leq (d+1)|V| - \binom{d+2}{2}.
\end{equation*}
As discussed in \cref{section:minimallygloballyrigid}, this bound is not tight for $|V| \geq d+3$, but again, $K_{d+1,n-d-1}$ with $n \geq \binom{d+2}{2} + 1$ gives an almost tight example.  
Kaszanitzky and Kir\'aly showed that the same upper bound holds for the number of edges in a minimally vertex-redundantly rigid graph \cite{KK}. In this case, a tight example can be obtained from $K_{d+1,n-d-1}$ by adding all possible edges between the vertices of the color class of size $d+1$.

Note that $K_{d+1,n-d-1}$ with $n \geq \binom{d+2}{2}+1$ is not only minimally globally rigid in $\RR^d$, but also minimally redundantly rigid in $\RR^d$. We believe that the number of edges in a minimally redundantly rigid graph in $\RR^d$ on vertex set $V$ is also at most $(d+1)|V|$, minus a small constant depending on $d$.

\subsection{Globally \texorpdfstring{$[k,d]$}{[k,d]}-rigid graphs}

\cref{bridgeconjspec} implies that minimally globally rigid graphs in $\RR^2$ are $\mathcal{R}_3$-independent. This can be generalized in the following way. Let us say that a graph $G$ is \emph{globally $[k,d]$-rigid} if it is globally rigid in $\RR^d$ and remains so after the deletion of any set of less than $k$ vertices. The graph is \emph{minimally globally $[k,d]$-rigid} if it is globally $[k,d]$-rigid but $G-e$ is not globally $[k,d]$-rigid, for every edge $e$ of $G$. \cref{bridgeconjspec} can be used to show that for any $k \geq 1$ and $d \in \{1,2\}$, if $G = (V,E)$ is minimally globally $[k,d]$-rigid, then $G$ is $\mathcal{R}_{d+k}$-independent. This implies that
\[|E| \leq (d+k)|V| - \binom{d+k+1}{2}.\]
The argument follows the proof of \cite[Lemma 6]{KK}. The upper bound is almost tight.
We omit the details.

\section{Acknowledgements}

We thank the anonymous reviewers for carefully reading the manuscript.
\cref{conjecture:minimallygloballyrigid} was first posed at the 2015 ``Advances in Combinatorial and Geometric Rigidity'' workshop hosted by the Banff International Research Station. 

This work was supported by the ÚNKP-21-3 New National Excellence Program of the Ministry for Innovation and Technology, the Hungarian Scientific Research Fund grant no. K135421
and the project Application Domain Specific Highly Reliable IT Solutions which has been
implemented with the support provided from the National Research, Development and
Innovation Fund of Hungary, financed under the Thematic Excellence Programme 
TKP2020-NKA-06 (National Challenges Subprogramme) funding scheme.

\appendix

\crefalias{subsection}{appsec}
\section*{Appendix}
\renewcommand{\thesubsection}{\Alph{subsection}}

\subsection{\texorpdfstring{$\Rd$}{Rd}-connectivity and \texorpdfstring{$2$}{2}-sums} \label{appendix:A}
In the following, we prove \cref{lemma:Mconnected2sum,lemma:redundantlyMconnected2sum}. We shall need the following folklore statement which is a consequence of the well-known ``Gluing Lemma'' of Whiteley \cite[Lemma 11.1.9]{whiteley_1996}.

\begin{lemma}\label{lemma:circuitseparatingedge}
If $G = (V,E)$ is an $\Rd$-circuit and $\{u,v\}$ is a separating vertex pair of $G$, then $uv \notin E$.
\end{lemma}

\begin{proof}[Proof of \cref{lemma:Mconnected2sum}]
\textit{(a)} $\Rightarrow$ \textit{(b)}: Let $e_i \in E_i$ be an edge in $G_i$ different from $u_iv_i$, for $i=1,2$. We shall show that there is an $\Rd$-circuit in $G_1 \twosum G_2$ containing $e_1$ and $e_2$. Indeed, from the assumption that $G_i$ is $\Rd$-connected, it contains an $\Rd$-circuit $C_i$ such that $e_i,u_iv_i \in C_i$. By \cref{lemma:circuit2sum}, $C_1 \twosum C_2$ is an $\Rd$-circuit of $G$ which contains $e_1$ and $e_2,$ as desired. This shows that any pair of edges $e_i \in E_i, i =1,2$ is in the same connected component of $\mathcal{R}_d(G_1 \twosum G_2)$, which implies that there is only one connected component, so $G_1 \twosum G_2$ must be $\Rd$-connected.

\textit{(b)} $\Rightarrow$ \textit{(c)}: It suffices to show that $uv$ is contained in some $\Rd$-circuit of $G_1 \twosum G_2 + uv$. As before, let $e_i \in E_i$ be an edge in $G_i$ different from $u_iv_i$, for $i=1,2$; then by $\Rd$-connectivity there is an $\Rd$-circuit $C$ in $G_1 \twosum G_2$ containing $e_1$ and $e_2$. By \cref{lemma:Mcircuits}, $C$ is $2$-connected, so we must have $u,v\in V(C)$. Let $C_1,C_2$ be the graphs obtained from a $2$-separation of $C$ along $\{u,v\}$. Then $C_1$ is a subgraph of $G_1 \twosum G_2 + uv$, and it is an $\Rd$-circuit by \cref{lemma:circuit2sum}, so $uv$ is indeed contained in an $\Rd$-circuit.

\textit{(c)} $\Rightarrow$ \textit{(a)}: By transitivity and symmetry, it suffices to show that for every edge $e_1$ of $G_1$ other than $u_1v_1$, there is an $\Rd$-circuit in $G_1$ containing $e_1$ and $u_1v_1$. Since $G_1 \twosum G_2 + uv$ is $\Rd$-connected, it contains such an $\Rd$-circuit $C$. By \cref{lemma:circuitseparatingedge}, $\{u,v\}$ cannot be a separating pair in $C$, so $C$ actually lies in $G_1$, as required. 
\end{proof}

The proof of \cref{lemma:redundantlyMconnected2sum} now follows from repeated applications of \cref{lemma:Mconnected2sum}.

\begin{proof}[Proof of \cref{lemma:redundantlyMconnected2sum}]
\textit{(a)}: By symmetry, it suffices to show that if $e_1$ is an edge of $G_1$ other than $u_1v_1$, then $G_1 \twosum G_2 - e_1$  is $\Rd$-connected. This follows from \cref{lemma:Mconnected2sum}, after noting that $G_1 \twosum G_2 - e_1 = (G_1 - e_1) \twosum G_2$, since by assumption both $G_1 - e_1$ and $G_2$ are $\Rd$-connected.

\textit{(b)}: This follows from the observation that if $e_1$ is an edge of $G_1$ other than $u_1v_1$, then by \cref{lemma:Mconnected2sum}, $(G_1 - e_1) \twosum G_2$ is $\Rd$-connected if and only if $(G_1 - e_1) \twosum G_2 + uv$ is $\Rd$-connected.

\end{proof}

\end{document}